\documentclass[11pt,a4paper]{amsart}

\usepackage[english]{babel}
\usepackage[utf8x]{inputenc}
\usepackage{amsmath}
\usepackage{xcolor}
\usepackage{tikz}
\usetikzlibrary{fit,shapes.geometric}
\usepackage{lmodern}
\usepackage[T1]{fontenc}
\usepackage{verbatim}
\usepackage{lmodern} 
\usepackage[a4paper,top=3cm,bottom=2cm,left=2.5cm,right=2.5cm,marginparwidth=1.75cm]{geometry}

\usepackage{amsmath}
\usepackage{tikz-cd}
\usepackage{graphicx}
\usepackage{quiver}
\usepackage[colorlinks=true, allcolors=blue]{hyperref}
\usepackage{amscd}     
\usepackage{amsfonts}
\usepackage{dsfont}
\usepackage{latexsym}
\usepackage{amscd}
\usepackage{verbatim}
\usepackage{graphicx}
	\numberwithin{equation}{section}
\usepackage{amssymb}
\usepackage{amsthm}
	\theoremstyle{plain}
		\newtheorem{thm}{Theorem}[section]
        \newtheorem{oliverthm}[thm]{G\"afvert's Theorem}
		\newtheorem{coro}[thm]{Corollary}
		\newtheorem{lem}[thm]{Lemma}
		\newtheorem{prop}[thm]{Proposition}
        
	\theoremstyle{definition}
		\newtheorem*{defn*}{Definition}
		\newtheorem*{goal*}{Goal}
		\newtheorem*{question*}{Question}

        \newtheorem{point}[thm]{}
	\theoremstyle{remark}

\usepackage{cleveref}        
\usepackage{mathtools} 
\usepackage{mathrsfs} 
\usepackage{eucal} 

\usepackage{microtype}
\usepackage{adjustbox}

\usepackage{mparhack}
	
\usepackage{xspace}

\usepackage{caption}
\usepackage{enumitem}

\usepackage{tikz} 
\usetikzlibrary{matrix,arrows,decorations.pathmorphing}
\usepackage{tikz-cd}
\usepackage{xcolor}
\definecolor{cherry}{rgb}{1.0, 0.604, 0.373}
\definecolor{softcherry}{RGB}{200, 60, 80} 
\definecolor{cherryblossom}{RGB}{220, 120, 140}

\newcommand{\longequiv}{%
  \mathrel{\vcenter{\offinterlineskip
    \hbox{\rule{1.8em}{0.35pt}}\kern0.32ex
    \hbox{\rule{1.8em}{0.35pt}}\kern0.32ex
    \hbox{\rule{1.8em}{0.35pt}}%
  }}%
}

\newcommand\myequivv{\stackrel{\mathclap{\tiny\mbox{6.6}}}{\longequiv}}

\newcommand\myequiv{\stackrel{\mathclap{\tiny\mbox{4.2}}}{\equiv}}
\newcommand\myequivvv{\stackrel{\mathclap{\tiny\mbox{5.4}}}{\longequiv}}
\newcommand\myequivvvv{\stackrel{\mathclap{\tiny\mbox{7.7}}}{\longequiv}}

\newcommand{\M}{\mathcal{M}}

\renewcommand{\=}{\coloneqq}	 

\newcommand\TODO[3]{\hbox to 0pt{\textcolor{#1}{$^\bullet$}}\marginpar{\footnotesize \textcolor{#1}{#2: #3}}}

\DeclareMathOperator{\vect}{\mathrm{vec}_K}

\usepackage[
  backend=biber,
  style=alphabetic
]{biblatex}

\title{Choosing the Geometry: valuations, norms, and contours. }

\author{Jens Agerberg}
\email{jensag@kth.se}
\author{Wojciech Chach\'olski}
\email{wojtek@kth.se}
\author{Christina Kapatsori}
\email{chrkap@kth.se}
\address{KTH, Mathematics,  10054 Stockholm, Sweden}

\begin{document}
\maketitle

\begin{abstract}
We develop a general framework for constructing pseudometrics on objects of Abelian categories using valuations, which assign non-negative costs to morphisms. Specializing to finitely presented vector space representations of the poset $[0,\infty)^\ell$, a natural setting for multiparameter persistence modules, we show that a large class of these distances is governed by contours, right-continuous lax actions of $[0,\infty)$ on the parameter poset. Properties of contours translate into geometric properties of the induced distances, including conditions under which the resulting pseudometrics are metrics and allowing us to identify a rich family of compact subsets of finitely presented representations.
\end{abstract}

\section{Introduction}

Finitely presented vector space representations of the poset $[0, \infty)^\ell$ have become standard for encoding homological  properties of data, such as point clouds. This makes the study of distances and the resulting geometric structures on the space of such representations a natural and important problem.
A substantial body of work has addressed this question; see, for example,~\cite{MR4172343, bauer2026metricallycompletekrullschmidtspace, PereaMunchKhasawneh2023, MR4588138, MR3413628, MR3735858}. Nevertheless, the focus has largely been restricted to a small collection of metrics, most notably the interleaving distance and its Wasserstein-type variants.
Our goal is to develop a much richer class of distances on finitely presented representations of $[0, \infty)^\ell$, capturing both the algebraic features of the representation category and the  geometry of the parameter space  $[0, \infty)^\ell$.

We discuss the algebraic aspect  in the more general setting of an arbitrary Abelian category. Analogously to how edge weights on a graph induce  distances between its vertices, we investigate how certain assignments of extended non-negative real values to morphisms in an Abelian category give rise to distances 
between its objects, by minimizing over zig-zags. We call such assignments \textbf{valuations} and identify algebraic properties of valuations that lead to reducing the minimization over arbitrary zig-zags to spans (Theorem~\ref{fgbfgbdfgbnn}).   
We also provide sufficient conditions under which the cokernel and direct-sum constructions are Lipschitz with respect to these distances (Proposition~\ref{tyjrytujt}).
Examples of valuations that satisfy all these requirements are those constructed using \textbf{norms}, which assign values to objects in a way that is compatible with exact sequences
(sub-additivity, see~\ref{adfsghdfgjh}).   By summing up the norms of the kernels and cokernels of  morphisms,  a valuation, called \textbf{homogeneous},  is obtained. 
Norms have been studied before under the name of \textbf{amplitudes} (see~\cite{MR4776401}) and are related to \textbf{noise systems} (see~\cite{MR3735858}).
 
Another important class of valuations consists of the \textbf{ultra valuations} (see~\ref{adfsgdfghjgukj}), which support a notion of approximation at a prescribed resolution $t$.
More precisely, for every $t\geq 0$ and every object $X$, we assume  there is the minimal sub-object $X[t]\subset X$ (called \textbf{$t$-shift}) whose inclusion has valuation at most $t$ (see~\ref{sdgdhjgk}). 
Our main algebraic result  (G\"afvert's Theorem~\ref{sadgfghjg}) is a classification of homogeneous ultra-valuations on the category of finitely presented representations of the poset $[0, \infty)^\ell$  for which the $t$-shift does not increase the number of generators: $|\beta_0(X[t])|\leq |\beta_0(X)|$ for every $t$; see \ref{dfhdfgjghkj}. 
These valuations are called \textbf{simple}, and G\"afvert's Theorem, named after~\cite[Theorem 9.6]{gafvert2021stableinvariantsmultiparameterpersistence},
exhibits a bijective correspondence between such valuations and \textbf{contours}. Contours 
are right-continuous lax actions of the monoid $[0,\infty)$ on the poset $[0, \infty)^\ell_\infty$; see~\ref{assfadfhfgs}. 
If a contour is such that 
 $0$ acts as the identity, then the associated distance 
 on representations is a metric, see Proposition~\ref{asdgadsfth}. Such contours are called \textbf{unital}.

 A convenient way to describe a contour is through a \textbf{step function} (Proposition~\ref{sdgsdfgh}), which is particularly useful for studying the associated distance. For example, Proposition~\ref{asfhgfgdjgh} provides an explicit formula for the distance between certain free functors in terms of the step function of the contour.

Enumerating simple homogeneous valuations via contours reveals many symmetries of both the space of these valuations and the distances they induce. For instance, there is a natural action of the group of order isomorphisms of  $[0,\infty)^\ell$ on the space of contours (see~\ref{sdgdghj}), and consequently on the space of simple homogeneous valuations. In Section~\ref{section6}, we describe several additional endomorphisms of the space of contours. Each of these gives rise to a corresponding transformation of simple homogeneous valuations and of the induced distances.

By restricting the distance associated with a contour to the free functors on one generator, one obtains a distance $d_\ast$ on the set  $[0, \infty)^\ell$. If the identity map $\text{id}\colon ([0, \infty)^\ell, d_{\text{Euc}})\to  ([0, \infty)^\ell, d_\ast)$
is continuous (respectively, Lipschitz), then the contour is called \textbf{continuous} (respectively, \textbf{Lipschitz}) (see~\ref{asvdgg}). A necessary condition for a contour $\ast$ to be continuous is that the strict inequality $(v\ast t)_i> v_i$ holds for every positive $t$
 and every $i$.
For example, the standard contour defined by $v\ast t=  v+(t,\ldots, t)$
(see~\ref{assfadfhfgs}) is Lipschitz. By definition, continuous and Lipschitz contours allow us to control the geometry of $[0, \infty)^\ell$ with respect to the associated contour distance using the standard Euclidean geometry. The main result of this article is that this control extends far beyond the free functors: we show how it can be used to study the geometry of the entire space of finitely presented representations of $[0, \infty)^\ell$. Our main result is 
Theorem~\ref{sadfhfggjh} which describes a rich family  of compact subsets of representations for such contours.

Acting by continuous  (respectively, Lipschitz) order isomorphisms of  $[0,\infty)^\ell$ preserve continuity (respectively, Lipschitz-ness) of contours, see Proposition~\ref{asfhgfgjghjkg}. One can then apply such isomorphisms to the standard contour to obtain a rich space of  Lipschitz and continuous    contours for which we can therefore  control the induced geometry on the space of finitely presented representations. 

The figure below graphically summarizes the narrative we have adopted for discussing various distances on finitely presented vector space representations of the poset $[0, \infty)^\ell$:

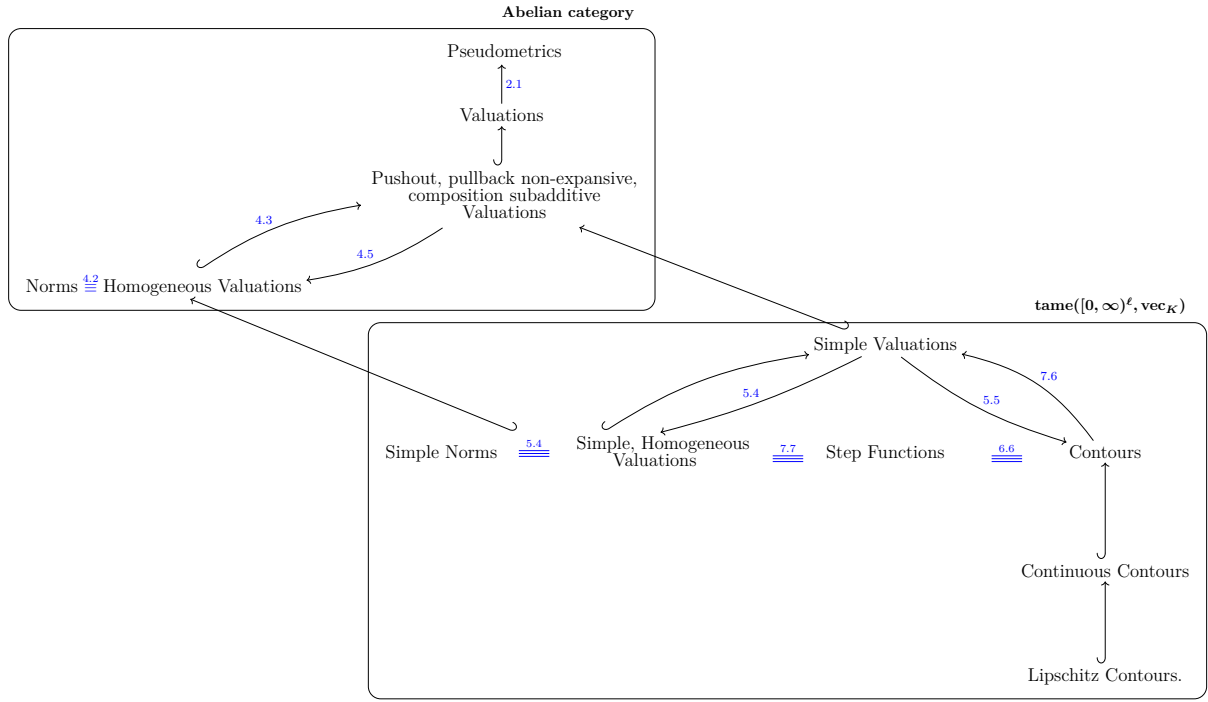
\[
\begin{adjustbox}{max width=\textwidth}
\tikzset{
  group two/.style={
    xshift=15mm,
    yshift=-13mm
  }
}
\tikzcdset{
  hook/.code={
    \pgfsetarrowsstart{tikzcd right hook[scale=2.5]}
  },
  hook'/.code={
    \pgfsetarrowsstart{tikzcd left hook[scale=2.5]}
  }
}

\begin{tikzcd}[
  column sep=large,
  row sep=large,
  cells={nodes={inner sep=1pt}},
  arrows={shorten <=2pt, shorten >=2pt},
execute at end picture={
  \node[
    draw,
    rounded corners=8pt,
    fit=(pseudo)(valuations)(properties)(norms),
    inner xsep=10pt,
    inner ysep=8pt
  ] (groupone) {};
  \node[
    draw,
    rounded corners=8pt,
    fit=(simpleval)(simplenorms)(stepfunctions)(contours)(contcontours)(lipcontours),
    inner xsep=10pt,
    inner ysep=8pt
  ] (grouptwo) {};
\node[
  above=3pt,
  anchor=south east,
  font=\small\bfseries,
  fill=white,
  inner sep=2pt
] at ([xshift=-4mm]groupone.north east)
  {Abelian category};
\node[
  above=3pt,
  anchor=south east,
  font=\small\bfseries,
  fill=white,
  inner sep=2pt
] at ([xshift=-4mm]grouptwo.north east)
  {$\boldsymbol{\mathrm{tame}([0,\infty)^\ell,\mathrm{vec}_K)}$};
}
]
&
|[alias=pseudo]| {\vcenter{\hbox{
\begin{tabular}{@{}c@{}}
\large Pseudometrics\\[-2pt]
\end{tabular}}}}
&
\\
&
|[alias=valuations]| {\text{\large  Valuations}}
&
\\
&
|[alias=properties]| {\vcenter{\hbox{
\begin{tabular}{@{}c@{}}\large
Pushout, pullback non-expansive,\\[-2pt]
\large composition subadditive\\[-2pt]
\large Valuations
\end{tabular}}}}
&
\\
|[alias=norms]|
\text{\large  Norms\  }{\hyperref[adgfadfhfg]{\myequiv}}
\text{\large \   Homogeneous Valuations}
&
&
|[alias=simpleval, group two]|
{\text{\large  Simple Valuations}}
\\
&
|[alias=simplenorms, group two]|
{\text{\large Simple Norms}\quad
\mathrel{\hyperref[dfhdfgjghkj]{\myequivvv}}
\vcenter{\hbox{
\begin{tabular}{@{}c@{}}
\quad \large Simple, Homogeneous\\[-2pt]
\large Valuations
\end{tabular}}}}
&
|[alias=stepfunctions, group two]|
{\text{\large Step Functions}}
&
|[alias=contours, group two]|
{\text{\large Contours}}
\\
& & &
|[alias=contcontours, group two]|
{\text{\large Continuous Contours}}
\\
& & &
|[alias=lipcontours, group two]|
{\text{\large Lipschitz Contours.}}
%
\arrow[
  hook,
  from=3-2,
  to=2-2
]
\arrow[
  from=2-2,
  to=1-2,
  "{\hyperref[adgsdfh]{2.1}}"'
]
%
\arrow[
  bend left=12,
  hook, 
  from=4-1,
  to=3-2,
  "{\hyperref[asgsdfh]{4.3}}"
]
\arrow[
  bend left=12,
  from=3-2,
  to=4-1,
  "{\hyperref[asdfdfghdfhsfgh]{4.5}}"'
]
\arrow[
  phantom,
  from=5-3,
  to=5-4,
  "\mathrel{\hyperref[sdgsdfgh]{\myequivv}}" description
]
%
\arrow[
  hook',
  from=5-2,
  to=4-1
]
\arrow[
  hook',
  from=4-3,
  to=3-2,
  shorten <=8pt
]
%
\arrow[
  phantom,
  from=5-2,
  to=5-3,
  "\mathrel{\hyperref[sadgfghjg]{\myequivvvv}}"
]
%
\arrow[
  bend left=12,
  hook,
  from=5-2,
  to=4-3
]
\arrow[
  bend left=6,
  from=4-3,
  to=5-2,
  "{\hyperref[dfhdfgjghkj]{5.4}}"'
]
%
\arrow[
  bend right=10,
  looseness=1.05,
  swap,
  from=4-3,
  to=5-4,
  "{\hyperref[asdgfdfhf]{5.5}}"'
]
\arrow[
  bend right=20,
  looseness=1.05,
  from=5-4,
  to=4-3,
  "{\hyperref[sadxcbgfgjh]{7.6}}"'
]
\arrow[
  hook,
  from=6-4,
  to=5-4
]
\arrow[
  hook,
  from=7-4,
  to=6-4
]
\end{tikzcd}
\end{adjustbox}
\]

 \bigskip

\noindent\textbf{Acknowledgements.} This work was partially supported by the Wallenberg AI, Autonomous Systems and
Software Program (WASP) funded by the Knut and Alice Wallenberg Foundation.

\section{Valuations and  distances on Abelian categories}\label{section2:val-dist}
\begin{point}\label{adgsdfh}
Let $\M$ be an Abelian category. 
Consider a function $\omega\colon\text{mor}(\M)\to [0,\infty]$ such that
$\omega(f)=0$ if $f$  is an isomorphism. Note that $\infty$ is allowed as  a value of  $\omega$. Such a function leads to a pseudo-metric:
for objects $X$ and $Y$ in $\M$ define
\[d_{\omega}(X,Y):=
\text{inf}\left\{\sum \omega(f_i) \mid X=
\begin{tikzcd}[column sep=small]W_0\ar{r}{f_1} &W_1 & W_2\ar{l}[swap]{f_2} \ar{r}{f_3 }& \cdots  & 
W_l=Y\ar{l}[swap]{f_l}\end{tikzcd}\right\}.\]
Indeed, from this definition it is clear that $d_\omega$ satisfies the following properties, for  all objects $X$, $Y$, and $Z$  in $\M$: 
(1)  $d_\omega(X,Y)=d_\omega(Y,X)$, (2) if $X$ and $Y$
 are isomorphic, then $d_\omega(X,Y)=0$, and (3) $d_\omega(X,Y) + d_\omega(Y,Z)\geq d_\omega(X,Z)$.

Consider a modification of $\omega$  to a new function: $\omega'(f):=\omega(f)$ if $f\not=0$, and $\omega'(0\colon X\to Y):=\text{min}\{\omega(0\colon X\to Y), \omega(0\colon Y\to X)\}$ for all objects $X$ and $Y$ in $\M$. This  new function is symmetric in the following sense: $\omega'(0\colon X\to Y)=\omega'(0\colon Y\to X)$ for all objects $X$ and $Y$.
Moreover, this modification leaves the induced distance unchanged:
$d_\omega(X,Y)=d_{\omega'}(X,Y)$. 

A function $\omega\colon\text{mor}(\M)\to [0,\infty]$ 
that satisfies  the mentioned two  requirements: (1) $\omega(f)=0$ if $f$ is an isomorphism, and (2) $\omega(0\colon X\to Y)=\omega(0\colon Y\to X)$ for all object $X$ and $Y$ in $\M$, is called \textbf{valuation} on $\M$.

For an arbitrary valuation  $\omega$, to calculate 
the distance $d_\omega(X,Y)$ one needs to search through   zig-zags  of arbitrary length connecting $X$ and $Y$. This process can be made more effective if  $\omega$ satisfies additional  properties.
\end{point}

\begin{point}\label{asgdhjfdh}
    Let $\omega$ be a valuation on  $\M$.
        In the following we refer to commutative squares in $\M$ of the form:
\begin{equation}
\begin{tikzcd}[column sep=small, row sep=small]
           X_0\ar{r}{f}\ar{d} & X_1\ar{d}\\
           Y_0\ar{r}{g} & Y_1
\end{tikzcd}\tag{$\ast$}
\end{equation}
    \begin{itemize}
        \item  $\omega$ is called \textbf{push-out non-expansive} if   $\omega(f)\geq \omega(g)$ for every commutative square ($\ast$) that is push-out.
        \item $\omega$ is called \textbf{pull-back non-expansive} if  $\omega(g)\geq \omega(f)$ for every commutative square ($\ast$) that is pull-back.
        \item $\omega$  is  called \textbf{composition sub-additive} if 
       $\omega(gf)\leq \omega(f) + \omega(g)$ for every pair of composable morphisms $f\colon X\to Y$ and $g\colon Y\to Z$.
       \item  $\omega$  is  called \textbf{monotone} if 
       $\omega(gf)\geq \text{max}\{\omega(f), \omega(g)\}$ for every pair of composable morphisms $f\colon X\to Y$ and $g\colon Y\to Z$ where either $f$ is an epimorphism or $g$ is a monomorphism.
    \end{itemize}
\end{point}
\begin{prop}
 Assume $\omega$
   is a  valuation on an abelian category $\M$ which is  composition sub-additive. Then $d_\omega(X,0)=\omega(X\to 0)$ for every object $X$
 in $\M$. 
 \end{prop}
 \begin{proof}
     The inequality $d_\omega(X,0)\leq \omega(X\to 0)$ is clear from the definition. To show the opposite inequality, we prove by induction on $l$
that  $\sum \omega(f_i)\geq \omega(X\to 0)$ for any zigzag:
\[
\begin{tikzcd}[column sep=small]X=W_0\ar{r}{f_1} &W_1 & W_2\ar{l}[swap]{f_2} \ar{r}{f_3 }& \cdots  & 
W_l=0\ar{l}[swap]{f_l}\end{tikzcd}.\]
The statement is clear if $l=1$. Assume $l\geq 2$.
Composition sub-additivity and the fact that $\omega$ is a valuation, imply: \[\omega(f_{l-1})+\omega(f_l)= \omega(f_{l-1}\colon W_{l-2}\to W_{l-1})+\omega(W_{l-1}\to 0)\geq \omega(W_{l-2}\to 0)= \omega(0\to W_{l-2}).\]
This, together with the inductive assumption, gives the desired inequality:
\[\sum \omega(f_i)\geq \sum_{i=1}^{l-2} \omega(f_i) + \omega(0\to W_{l-2})\geq  \omega(X\to 0).\qedhere\]
\end{proof}
\begin{prop}\label{sdgdsfhjf}
    Assume $\omega$
   is a  valuation on an abelian category $\M$ which is push-out  non-expansive and 
   composition sub-additive. 
   \begin{enumerate}
        \item Then, $\omega(f\oplus g)\leq \omega(f)+\omega(g)$ for all morphisms $f$ and $g$ in $\M$.
        \item Consider the following commutative diagram in $\M$ with  exact rows:
        \[\begin{tikzcd}
           X_0\ar{r}{f}\ar{d}[swap]{\varphi_0} & X_1\ar{d}{\varphi_1}\ar[two heads]{r} & \text{\rm coker}(f)\ar{d}{\text{\rm coker}(\varphi_0,\varphi_1)}\ar{r} & 0\\
            Y_0\ar{r}{g} & Y_1
            \ar[two heads]{r} & \text{\rm coker}(g)\ar{r} & 0.
    \end{tikzcd}\]
    Then, $\omega(\text{\rm coker}(\varphi_0,\varphi_1))\leq \omega(\varphi_0)+\omega(\varphi_1)$.
   \end{enumerate}
\end{prop}
\begin{proof}
    \noindent
    (1):\quad  Since the following squares are push-outs,
    $\omega(f\oplus \text{id}_{Y_0})\leq \omega(f)$ and 
    $\omega(\text{id}_{X_1}\oplus g)\leq \omega(g)$.
    \[\begin{tikzcd}[column sep=15mm]
        X_0\ar{r}[description]{\begin{bsmallmatrix}
        1\\ 0
        \end{bsmallmatrix}}\ar{d}[swap]{f} & X_0\oplus Y_0\ar{d}{f\oplus \text{id}_{Y_0}}\\
        X_1\ar{r}[description]{\begin{bsmallmatrix}
        1\\ 0
        \end{bsmallmatrix}} & X_1\oplus Y_0
    \end{tikzcd}
    \ \ \ \ \ \ \ \ 
    \begin{tikzcd}[column sep=15mm]
        Y_0\ar{r}[description]{\begin{bsmallmatrix}
        0\\ 1
        \end{bsmallmatrix}}\ar{d}[swap]{g} & X_1\oplus Y_0\ar{d}{\text{id}_{X_1}\oplus g}\\
        Y_1\ar{r}[description]{\begin{bsmallmatrix}
        0\\ 1
        \end{bsmallmatrix}} & X_1\oplus Y_1
    \end{tikzcd}\]
    Consequently, $\omega(f\oplus g)=\omega((\text{id}_{X_1}\oplus g) \circ(f\oplus \text{id}_{Y_0}))\leq \omega(f\oplus \text{id}_{Y_0}) +\omega(\text{id}_{X_1}\oplus g)\leq \omega(f)+\omega(g)$.
    \smallskip

    \noindent
    (2):\quad Consider the induced commutative diagram where the top left square is a push-out:
    \[
    \begin{tikzcd}
        X_1\ar[two heads]{r}\ar{d}{\varphi_1} & \text{\rm coker}(f)\ar{d}{\overline{\varphi_1}}\ar[bend left=20pt]{rd}{\text{\rm coker}(\varphi_0,\varphi_1)} \\
        Y_1\ar{r}\ar[bend right=25pt, two heads]{rr} & \text{\rm coker}(\varphi_1 f = g\varphi_0)\ar{r}{h} & \text{\rm coker}(g)
    \end{tikzcd}
    \]
    Then $\omega(\overline{\varphi_1})\leq \omega(\varphi_1)$ and  hence
    $\omega(\text{\rm coker}(\varphi_0,\varphi_1))\leq \omega(\varphi_1)+ \omega(h)$.
    Consider next the following push-out squares:
    \[
    \begin{tikzcd}
        \text{coker}(\varphi_0)\ar{r}\ar{d} & \text{\rm coker}(g\varphi_0)\ar{d}{h}\\
        0 \ar{r} & \text{\rm coker}(g)
    \end{tikzcd}\ \ \ \ \ \ \ \ \ 
    \begin{tikzcd}
        X_0\ar{d}{\varphi_0} \ar{r} & 0\ar{d}\\
        Y_0\ar{r} & \text{coker}(\varphi_0)
    \end{tikzcd}
    \]
    They give  $\omega(h)\leq \omega(\text{coker}(\varphi_0)\to 0)=
    \omega(0\to \text{coker}(\varphi_0))\leq \omega(\varphi_0)$, and the statement follows. 
\end{proof}

We use the properties defined in~\ref{asgdhjfdh} to 
specify valuations for 
which calculating associated distances can be done more effectively, without searching through zig-zags of arbitrary length:
\begin{thm}\label{fgbfgbdfgbnn}
    Let   $\omega$ be a valuation on  an abelian category $\M$.
    \begin{enumerate}
        \item If   $\omega$ is push-out
        non-expansive and composition sub-additive, then:
        \[d_\omega(X,Y)= \text{\rm inf}\{\omega(f)+\omega(g) \mid   \begin{tikzcd}[column sep=small]X\ar{r}{f} &W  & Y\ar{l}[swap]{g}\end{tikzcd}\}.\]
        \item If   $\omega$ is push-out
        non-expansive, composition sub-additive, and monotone, then:
        \[d_\omega(X,Y)= \inf\Bigl\{
\omega(f)+\omega(g)
\;\Bigm|\;
\text{for epimorphisms }
\begin{tikzcd}[baseline=-0.6ex, column sep=small, ampersand replacement=\&]
X\oplus Y
  \arrow[two heads, r, "{\begin{bsmallmatrix} f & g \end{bsmallmatrix}}"]
\&
W
\end{tikzcd}
\Bigr\}.\]
        \item  If   $\omega$ is pull-back
        non-expansive and composition sub-additive, then:
        \[d_\omega(X,Y)= \text{\rm inf}\{\omega(f)+\omega(g) \mid   \begin{tikzcd}[column sep=small]X &W \ar{r}{g} \ar{l}[swap]{f}& Y\end{tikzcd}
        \}.\]
        \item 
        If   $\omega$ is pull-back
        non-expansive, composition sub-additive, and monotone,  then:
        \[d_\omega(X,Y)= \inf\Bigl\{\omega(f)+\omega(g) \;\Bigm|\;   \text{for monomorphisms}
        \begin{tikzcd}[column sep=40pt]W\ar[hook]{r}[description]{\begin{bsmallmatrix}f \\ g\end{bsmallmatrix}} & X\oplus Y\end{tikzcd}
        \Bigr\}.\]
   
    \end{enumerate}
\end{thm}

\begin{proof}
    By using push-outs and compositions, any zig-zag connecting $X$ and $Y$ can be converted into a diagram $X\rightarrow W\leftarrow Y$. 
    Since under the assumptions of (1), these operations do not increase the values of the valuations of the involved morphisms, thus statement (1) holds. 
    
    To prove (2), note that, if in addition $\omega$ is monotone, then  by taking the image  of the induced morphisms $X\oplus Y\to W$,  we obtain an epimorphism 
    without again increasing the values of  the valuations of the involved morphisms.

    Analogous arguments (involving pull-back, compositions and monotonicity) can be used to prove  statements (3) and (4). \end{proof}

\begin{point}\label{afssdfhdghk}
    Let $S$ be a set of $n$ elements. We think of it as a discrete category with only the identity morphisms. Recall that the functor category $\text{Fun}(S,\M)$, of $\M$ valued functors indexed by $S$, is also Abelian (see for example~\cite{freyd1964abelian}). Its objects can be identified with  sequences $(X_s)_{s\in S}$ of   objects in $\M$ and its morphisms with 
    sequences $(f_s)_{s\in S}$  of  morphisms in $\M$. 
    
    Let  $\omega$ be a valuation on $\M$.
    Then the function that assigns to a morphism $(f_s)_{s\in S}$ in
    $\text{Fun}(S,\M)$ the sum $\sum_{s\in S}\omega(f_s)$ is a valuation on $\text{Fun}(S,\M)$ which we also denote by $\omega$.  Since push-outs, pull-backs, and compositions in $\text{Fun}(S,\M)$ are formed index-wise, if $\omega$ on $\M$ is push-out/pull-back non-expansive, or composition sub-additive, or monotone, then so is $\omega$ on  $\text{Fun}(S,\M)$. The induced distance $d_\omega$ on the objects of 
    $\text{Fun}(S,\M)$ is the Manhattan extension of the distance $d_\omega$ on the objects of $\M$: $d_\omega\left((X_s)_{s\in S}, (Y_s)_{s\in S}\right)=\sum_{s\in S}d_\omega(X_s,  Y_s)$.
\end{point}

\begin{point}\label{iphfinpxla}
    Let $[1]$ be the category associated with the poset $0\leq 1$. Recall that, as in~\ref{afssdfhdghk}, the functor category $\text{Fun}([1], \M)$ is  Abelian.
    Its objects can be identified with morphisms in $\M$, and its morphisms $\varphi\colon f\to g$ with commutative
    squares:
    \[
    \begin{tikzcd}[column sep=small, row sep=small]
        X_0\ar{r}{f}\ar{d}[swap]{\varphi_0} & X_1\ar{d}{\varphi_1}\\
        Y_0\ar{r}{g} & Y_1.
    \end{tikzcd}
    \]
    
    Let $\omega$ be a valuation on $\M$. Then the function that assigns to $(\varphi_0,\varphi_1)\colon f\to g$ the sum $\omega(\varphi_0)+\omega(\varphi_1)$, is a valuation on $\text{Fun}([1], \M)$, which we also denote by $\omega$. As before, since push-outs, pull-backs, and compositions in $\text{Fun}([1],\M)$ are formed index-wise, if $\omega$ on $\M$ is push-out/pull-back non-expansive, or composition sub-additive, or monotone, then so is $\omega$ on  $\text{Fun}([1],\M)$. 
\end{point}

\begin{prop}\label{tyjrytujt}
    Let $\omega$  be a push-out non-expansive and composition sub-additive valuation on an abelian category $\M$. 
    \begin{enumerate}
        \item  Let $S$ be a set of $n$ elements. The functor $\oplus\colon \text{\rm Fun}(S,\M)\to \M$, $(f_s)_{s\in S} \mapsto \bigoplus_{s\in S}f_s$, induces  a 1-Lipschitz function on objects with respect to the distance $d_\omega$.
        \item The cokernel functor $\text{\rm coker}\colon\text{\rm Fun}([1],\M)\to \M$, $(\varphi_0,\varphi_1)\mapsto \text{\rm coker}(\varphi_0,\varphi_1)$, 
        induces  a 1-Lipschitz function on objects with respect to the distance $d_\omega$.
    \end{enumerate}
\end{prop}
\begin{proof}
    Statement (1) is a direct consequence of Proposition~\ref{sdgdsfhjf}.(1), 
    which gives 
    $\omega((f_s)_{s\in S})\geq \omega(\bigoplus_{s\in S}f_s)$ 
    for all morphisms $(f_s)_{s\in S}$ in $\text{\rm Fun}(S,\M)$. 
    
    Similarly, statement (2) is a direct consequence of Proposition~\ref{sdgdsfhjf}.(2), 
    which gives 
    $\omega(\varphi_0,\varphi_1)\geq \omega(\text{coker}(\varphi_0,\varphi_1))$ 
    for all morphisms $(\varphi_0,\varphi_1)$ in $\text{\rm Fun}([1],\M)$. 
\end{proof}

\section{Ultra valuations on Abelian categories}\label{sec3:ultra}
\begin{point}\label{sdgdhjgk}
Let $\omega$ be a valuation on an Abelian category $\M$. For $t$ in  $[0,\infty)$ and for an object $X$ in $\M$, 
a monomorphism $f\colon A\hookrightarrow X$ is called \textbf{$t$-minimal} (with respect to $\omega$) if: (1) $\omega(f)\leq t$, and (2) for every commutative diagram in $\M$ of the form:
    \begin{equation}\begin{tikzcd}[column sep=small, row sep=small]
        B\ar[bend right=25pt]{dr}{g}\ar{rr}{h} & & A\ar[hook', bend left=25pt]{dl}[swap]{f}\\
        & X
\end{tikzcd}\tag{$\ast\ast$}
    \end{equation}
if $g$ is a monomorphism with $\omega(g)\leq t$, then $h$ is an isomorphism.   
For example, $0\to X$ is $t$-minimal if, and only if,   $t\geq \omega(0\to X)$.

One might ask if such minimal monomorphisms exist and  how unique  they are. Uniqueness is addressed in:
\end{point}
\begin{prop}\label{asgdsghdfgjh}
Let $\omega$ be a pull-back non-expansive valuation on an Abelian category $\M$. Assume
$\omega(f\oplus g) \leq \text{\rm max}\{\omega(f),\omega(g)\}$ for all monomorphisms $f$ and $g$ in $\M$. If, for $t$ in $[0,\infty)$, $f\colon A\hookrightarrow X$ is a $t$-minimal monomorphism, then, for every 
monomorphism $g\colon B\hookrightarrow X$ with $\omega(g)\leq t$, there is a unique morphism $h\colon A\to B$ (necessarily a monomorphism) making the following diagram commutative:
  \[\begin{tikzcd}[column sep=small, row sep=small]
        A\ar[hook, bend right=25pt]{dr}{f}\ar{rr}{h} & & B\ar[hook', bend left=25pt]{dl}[swap]{g}\\
        & X
\end{tikzcd}\]
If in addition $g$ is $t$-minimal, then the morphism $h$ is an isomorphism.
\end{prop}

\begin{proof}
    Consider the following pull-back square in $\M$:
    \[\begin{tikzcd}[row sep = 30pt]
        A\cap B\ar[hook]{r}{p}\ar{d}[swap]{\begin{bsmallmatrix}
        \alpha\\ \beta
        \end{bsmallmatrix}} & X\ar{d}{\begin{bsmallmatrix}
        1\\ 1
        \end{bsmallmatrix}}\\
        A\oplus B\ar[hook]{r}{f\oplus g} & X\oplus X
    \end{tikzcd}\]
    By the pull-back non-expansiveness, 
    $\omega(p)\leq \omega(f\oplus g)\leq  \text{\rm max}\{\omega(f),\omega(g)\} \leq t$. Since $p$ is also a monomorphism, the minimality of $f$, implies that  $\alpha \colon A\cap B\to A$ is an isomorphism. Its inverse, composed with $\beta$ is the desired $h$. 
     The  uniqueness of $h$ is a consequence of the fact that all considered morphisms are monomorphisms. If in addition $g$ is $t$-minimal, then by the same argument $\beta$  is also an isomorphism, and hence so is $h$.
\end{proof}

\begin{point}\label{adfsgdfghjgukj}
A valuation $\omega$ on $\M$ is called \textbf{ultra} if:
(0) it is pull-back non-expansive, (1)
$\omega(f\oplus g) \leq \text{\rm max}\{\omega(f),\omega(g)\}$ for all monomorphisms $f$ and $g$ in $\M$, and (2) for every $t$ in $[0,\infty)$ and every object $X$ in $\M$, there is a $t$-minimal monomorphism $f\colon A\hookrightarrow X$. 

Assume $\omega$ is an ultra valuation on $\M$.
For every $t$ in $[0,\infty]$ and every object $X$ in $\M$, let us choose a $t$-minimal monomorphism $\text{T}_{X}[t]\colon X[t]\hookrightarrow X$. According to~\ref{asgdsghdfgjh}, such a minimal monomorphism  is 
unique up to a necessarily unique isomorphism. We refer to  the morphism $\text{T}_X[t]$ and the object 
$X[t]$ as the \textbf{$t$-shifts} of $X$. The assignment $X\mapsto  X[t]$ can be promoted to a functor for which $\text{T}_{X}[t]$ is a natural transformation.
\end{point}

\begin{prop}\label{sdgdfh}
    Let $\omega$ be an ultra valuation  on an Abelian category $\M$. For every
    morphism $f\colon X\to Y$ in $\M$ and every $t$ in $[0,\infty)$, there is a unique morphism $f[t]\colon X[t]\to Y[t]$ for which the following square commutes:
    \[
    \begin{tikzcd}
        X[t]\ar[hook]{r}{\text{\rm T}_{X}[t]}\ar{d}[swap]{f[t]} & X\ar{d}{f}\\
        Y[t]\ar[hook]{r}{\text{\rm T}_{Y}[t]} & Y.
    \end{tikzcd}
    \]
\end{prop}
\begin{proof}
    Consider a pull-back square:
    \[
    \begin{tikzcd}
        P\ar[hook]{r}{p}\ar{d} & X\ar{d}{f}\\
        Y[t]\ar[hook]{r}{\text{\rm T}_{Y}[t]} & Y.
    \end{tikzcd}
    \]
    Then $p$ is a monomorphism for which $\omega(p)\leq\omega(\text{\rm T}_{Y}[t])\leq t$. Thus, according to~\ref{asgdsghdfgjh},  there is a unique morphism $X[t]\to P$
    whose composition with $p$ is  $\text{\rm T}_{X}[t]$. 
    The composition of this morphism with $P\to Y[t]$ is the desired $f[t]$.
\end{proof}

The $t$-shift functor has the following properties. 

\begin{prop}\label{sadgsdfhsfgh}
    Let $\omega$ be an ultra valuation  on  an Abelian category $\M$.
    \begin{enumerate}
        \item  The natural morphism $X[t]\oplus Y[t]\to(X\oplus Y)[t]$
        is an isomorphism for all objects  $X$ and $Y$ in $\M$ and   all $t$ in $[0,\infty)$.
        \item For every object $X$ in $\M$ and every $s\leq t$ in $[0,\infty)$, there is a monomorphism $X[t]\hookrightarrow X[s]$  making the following diagram commutative:
        \[\begin{tikzcd}
            X[t]\ar[hook]{rr}\ar[hook, bend right=28pt]{dr}[pos=0.4]{T_X[t]} & & 
            X[s]\ar[hook',bend left=28pt]{dl}[ pos=0.4, swap]{T_X[s]}\\
            & X
        \end{tikzcd}\]
        \item If in addition to being ultra, $\omega$ is  composition sub-additive (see~\ref{asgdhjfdh}), then, for every object $X$ in $\M$ and every $s, t$ in $[0,\infty)$, there is a monomorphism $X[s+t]\hookrightarrow(X[s])[t]$ making the following diagram commutative:
        \[
        \begin{tikzcd}[column sep =35pt]
            X[s+t]\ar[hook]{r}\ar[hook, bend right=25pt]{dr}[pos=0.4]{T_X[s+t]} & (X[s])[t]\ar[hook]{r}{T_{X[s]}[t]} & X[s]\ar[hook', bend left=25pt]{dl}[pos=0.4, swap]{T_X[s]}\\
            & X
        \end{tikzcd}
        \]
        \item If in addition to being ultra, $\omega$ is monotone (see~\ref{asgdhjfdh}), then,  a monomorphism $f\colon Y\hookrightarrow X$ is such that $\omega(f)\leq t$ if, and only if, the shift $T_X[t]$ factors through $f$, i.e. if 
        there is a monomorphism  $X[t]\hookrightarrow Y$ making the following diagram commutative:
        \[
        \begin{tikzcd}
            X[t]\ar[hook]{r}\ar[hook, bend right = 40pt]{rr}{T_X[t]} & Y\ar[hook]{r}{f} & X
        \end{tikzcd}
        \]
    \end{enumerate}
\end{prop}

\begin{proof}
    \noindent

    \noindent
    (1):\quad The natural morphism  $X[t]\oplus Y[t]\to(X\oplus Y)[t] $ fits into the following commutative diagram where the indicated morphisms are monomorphism and the morphism indicated by the dotted arrow is given by~\ref{asgdsghdfgjh} as the ultra assumption on $\omega$ implies that $\omega(\text{T}_X[t]\oplus\text{T}_Y[t])\leq t$:
    \[
    \begin{tikzcd}
        X[t]\oplus Y[t]\ar[hook, bend right=25pt]{dr}[swap]{\text{T}_X[t]\oplus\text{T}_Y[t]}\ar{rr} & & (X\oplus Y)[t]\ar[hook', bend left=25pt]{dl}{\text{T}_{X\oplus Y}[t]}\ar[bend right=25pt, dotted]{ll}\\
        & X\oplus Y 
    \end{tikzcd}
    \]
    The morphism given by the dotted arrow is the desired inverse of 
    $X[t]\oplus Y[t]\to(X\oplus Y)[t] $.
    \smallskip
    
    \noindent
    (2,3):\quad Since $\omega(\text{T}_X[s])\leq s\leq t$, the existence of   the desired $X[t]\hookrightarrow X[s]$ in statement (2) is guaranteed by Proposition~\ref{asgdsghdfgjh}. The same argument works for statement (3) as by the composition sub-additivity assumption we have 
    $\omega(T_X[s]\circ T_{X[s]}[t])\leq s+t$.
    \smallskip
    
    \noindent
    (4):\quad If $\omega(f)\leq t$, then the desired monomorphism $X[t]\hookrightarrow Y$ is given by Proposition~\ref{asgdsghdfgjh}. 
    If such a factorization exists, then by monotonicity, since $f$ is a monomorphism,   $\omega(f)\leq \omega(T_X[t]) \leq t$. 
\end{proof}

\section{Norms on Abelian categories}\label{sec4:norms}
\begin{point}\label{adfsghdfgjh}
    The key examples of valuations satisfying all the assumptions appearing in Theorem~\ref{fgbfgbdfgbnn} are constructed using norms.
    A \textbf{norm} on $\M$ is a function $N\colon \text{ob}(\M)\to [0,\infty]$ satisfying the following properties: (1) if $0$ is a zero object in $\M$, then $N(0)=0$, and (2)
    if $0\to X\to Y\to Z\to 0$ is an exact sequence in $\M$, then
    $\text{max}\{N(X), N(Z)\}\leq N(Y)\leq N(X)+N(Z)$ (compare with~\cite{MR4776401} where the name \textbf{amplitudes} is used). In particular, it follows, that if $X$ and $Y$ are isomorphic, then $N(X)=N(Y)$.
    
    To specify a norm is the same as to specify a noise system as defined in~\cite{MR3735858}. Recall that a noise system is  a sequence 
    $\{\mathcal{S}_t\}_{t\in [0,\infty)}$ of collections of objects in $\M$
    such that: (1) the zero object  $0$  in $\M$ belongs to $\mathcal{S}_t$ for every $t$, (2) if $s\leq t$, then $\mathcal{S}_s\subset \mathcal{S}_t$,  (3) for every exact sequence  $0\to X\to Y\to Z\to 0$ in $\M$, if $Y$ is in $ \mathcal{S}_t$, then $X$ and $Z$ are in $\mathcal{S}_t$, and if $X$ is in $\mathcal{S}_t$ and $Z$ is in  $\mathcal{S}_s$, then $Y$ is in $ \mathcal{S}_{t+s}$, and (4) 
    $\mathcal{S}_t=\cap_{t<s}\mathcal{S}_s$ for every $t$ in $[0,\infty)$. The  noise system
    associated to a norm $N$ is given by $\mathcal{S}_t:=\{X \mid N(X)\leq t\}$.
    The norm associated to a noise system $\{\mathcal{S}_t\}_{t\in [0,\infty)}$ is given  by
    $N(X)=\text{inf}\{t\in[0,\infty)\  |\ X\in \mathcal{S}_t\}$.
\end{point}
\begin{point}\label{adgfadfhfg}
    Let $N$ be a norm on $\M$. For a morphism $f$ in $\M$, define $N(f):= N(\text{ker}(f)) + N(\text{coker}(f))$. In particular, if $f$ is an isomorphism, then $N(f)=0$. 
    Moreover, $N(0\colon X\to Y)=N(X)+N(Y)=N(0\colon Y\to X)$.
    Thus the obtained  function 
    $N\colon\text{mor}(\M)\to [0,\infty]$ is a valuation, which we call \textbf{associated} with the norm $N$. Valuations associated with  norms are called \textbf{homogeneous}.
    Note that if $\omega$ is a homogeneous valuation on $\M$, associated with the norm $N$, then
    $N(X)=\omega(0\to X)$. Thus a homogeneous valuation is associated with a unique norm.
\end{point}

\begin{prop}\label{asgsdfh}
    Homogeneous valuations on an Abelian category are  push-out and pull-back non-expansive, composition sub-additive, and monotone
    (see~\ref{asgdhjfdh}).
\end{prop}
\begin{proof}
    Let $N$ be a homogeneous valuation on $\M$, associated with a norm denoted by the same symbol $N$.
    Consider the following  commutative diagram in $\M$ with exact rows: 
    \[\begin{tikzcd}
           0\ar{r} &\text{ker}(f)\ar{r} \ar{d}[swap]{k} & X_0\ar{r}{f}\ar{d} & X_1\ar{d}\ar{r} & \text{coker}(f)\ar{d}{c}\ar{r} & 0\\
             0\ar{r} &\text{ker}(g)\ar{r} & Y_0\ar{r}{g} & Y_1
            \ar{r} & \text{coker}(g)\ar{r} & 0
    \end{tikzcd}\]
    If the middle square is a push-out, then the morphism $k$ is an epimorphism and $c$ is an isomorphism. Thus in this case $N(\text{ker}(f))\geq N(\text{ker}(g))$ and $N(\text{coker}(f))= N(\text{coker}(g))$, and consequently $N(f)\geq N(g)$.
    If this square is a pull-back, then the morphism $c$ is a monomorphism and $k$ is an isomorphism. Thus in this case 
    $N(\text{coker}(g))\geq N(\text{coker}(f))$ and $N(\text{ker}(g))= N(\text{ker}(f))$, and consequently $N(g)\geq N(f)$. Thus $N$ is both push-out and pull-back non-expansive. 

    Composable morphisms $f\colon X\to Y$ and $g\colon Y\to Z$ in $\M$ lead to two exact sequences:
    \[
    \begin{tikzcd}[column sep=small]
    0\ar{r} &\text{ker}(f)\ar{r} & \text{ker}(gf)\ar{r}{f'} &  \text{ker}(g)
    \end{tikzcd}
    \ \ \ \ \ \ \ \ 
    \begin{tikzcd}[column sep=small]
    \text{coker}(f)\ar{r}{g'} & \text{coker}(gf)\ar{r}{} &  \text{coker}(g)\ar{r} & 0
    \end{tikzcd}
    \] 
    Thus:
    \[N(\text{ker}(gf))\leq N(\text{ker}(f)) + N(\text{im}(f'))\leq
    N(\text{ker}(f)) + N(\text{ker}(g))\] 
    \[N(\text{coker}(gf))\leq  N(\text{im}(g')) + N(\text{coker}(g)) \leq
    N(\text{coker}(f)) + N(\text{coker}(g))\]
    By adding these inequalities, we get the composition sub-additivity $N(gf)\leq  N(f) + N(g)$. 

    If $f$ is an epimorphism, then $f'\colon \text{ker}(gf)\to \text{ker}(g)$ is also an epimorphism, and $\text{coker}(f)=0$.
    This gives $N(\text{ker}(gf))\geq \text{max}\{N(\text{ker}(f)), N(\text{ker}(g))\}$ and $N(\text{coker}(gf))=N(\text{coker}(g))$. All this together imply $N(gf)\geq \text{\rm max}\{N(f), N(g)\}$.
    
    If $g$ is a monomorphism, then so is $g'\colon\text{coker}(f)\to \text{coker}(gf)$ and $\text{ker}(g)=0$. This gives
    $N(\text{coker}(gf))\geq \text{max}\{N(\text{coker}(f)), N(\text{coker}(g))\}$ and $N(\text{ker}(gf))=N(\text{ker}(f))$.
    All this together imply $N(gf)\geq \text{\rm max}\{N(f), N(g)\}$, and we get the monotonicity of  $N$.
\end{proof}

\begin{prop}\label{adsgdsjgk}
Assume  $\omega$ is a valuation on an abelian category $\M$ which is push-out,  pull-back non-expansive, and composition sub-additive. 
    \begin{enumerate}
        \item Then $N_\omega(X):=\omega(0\to X)$ is a norm on $\M$.
        \item If $f$ is a monomorphism  in $\M$, then $\omega(f)=N_\omega(\text{\rm coker}(f))$.
        \item If $f$ is an epimorphism  in $\M$, then $\omega(f)=N_\omega(\text{\rm ker}(f))$.
        \item  For every
        morphism $f$ in $\M$:
        \[\text{\rm max}\{N_\omega(\text{\rm ker}(f)), N_\omega(\text{\rm coker}(f))\}\ \stackrel{(a)}{\leq}\  \omega(f)\  \stackrel{(b)}{\leq} \ 
        N_\omega(\text{\rm ker}(f)) + N_\omega(\text{\rm coker}(f))\ \stackrel{(c)}{\leq} \ 2\omega(f).\]
        \item 
        $  d_{\omega}\leq d_{N_\omega}\leq  2d_{\omega}$.
    \end{enumerate}
\end{prop}
\begin{proof}
    (1):\quad 
    The equality $N_\omega(0)=0$ is clear from the definition. Consider an exact sequence $0\to X\to Y\to Z\to 0$ in $\M$.
    This sequence leads to three commutative squares, the first one being 
    pull-back, the second being both push-out and pull-back, and the third  being  push-out:
    \[
    \begin{tikzcd}[column sep=small, row sep=small]
        0\ar{r}\ar{d} & 0\ar{d}\\
        X\ar{r} & Y
    \end{tikzcd}\ \ \ \ \ \ 
    \begin{tikzcd}[column sep=small, row sep=small]
        X\ar{r}\ar{d} & Y\ar{d}\\
        0\ar{r} & Z
    \end{tikzcd}\ \ \ \ \ \ 
    \begin{tikzcd}[column sep=small, row sep=small]
        Y\ar{r}\ar{d} & Z\ar{d}\\
        0\ar{r} & 0
    \end{tikzcd}
    \]
    By the push-out and  pull-back non-expansiveness we get:
    \begin{itemize}
        \item using the first square: $N_\omega(Y)=\omega(0\to Y)\geq \omega(0\to X)=N_\omega(X)$;
        \item using the third square: $N_\omega(Y)=\omega(0\to Y) = \omega(Y\to 0)\geq \omega(Z\to 0)=N_\omega(Z)$;
        \item using the second square: $N_\omega(Z)=\omega(0\to Z)=\omega (X\to Y)$.
    \end{itemize}
Thus applying the composition sub-additivity to $0\to X\to Y$, gives:
\[N_\omega(Y)=\omega(0\to Y)\leq \omega(0\to X)+\omega(X\to Y) = N_\omega(X)+N_\omega(Z)\]

\noindent
(2, 3):\quad 
If $f\colon X\hookrightarrow Y$ is a  monomorphism and $g\colon Y\twoheadrightarrow Z$ is an epimorphism, then we can form the following commutative diagrams that are both push-outs and pull-backs. 
\[
\begin{tikzcd}[row sep= small ]
    X\ar[hook]{r}{f}\ar{d} & Y\ar{d}\\
    0\ar{r} & \text{coker}(f)
\end{tikzcd}\ \ \ \ \ \ \ \ \ \ 
\begin{tikzcd}[row sep= small ]
    \text{ker}(g)\ar{r}\ar{d} & 0\ar{d}\\
    Y\ar[two heads]{r}{g} & Z
\end{tikzcd}
\]
The statements then follow from the assumption that $\omega$ is push-out and pull-back non-expansive.

\noindent
(4):\quad  Choose a morphism $f\colon X\to Y$ in $\M$ and consider the following commutative diagrams:
    \[
    \begin{tikzcd}[column sep=small, row sep=small]
        \text{ker}(f)\ar{r}\ar[hook]{d} & 0\ar{d}\\
        X\ar{r}{f} & Y
    \end{tikzcd}\ \ \ \ \ \ 
    \begin{tikzcd}[column sep=small, row sep=small]
        X\ar{r}{f}\ar{d} & Y\ar[two heads]{d}\\
        0 \ar{r} & \text{coker}(f)
    \end{tikzcd}\ \ \ \ \ \ 
    \begin{tikzcd}[column sep=small, row sep=small]
        X\ar[two heads]{r} \ar[bend right=40pt]{rr}[description]{f} & \text{im}(f)\ar[hook]{r} & Y
    \end{tikzcd}
    \]
    Since the first square is a pull-back, $\omega(f)\geq N_\omega(\text{ker}(f))$. The second square is a push-out and thus  $\omega(f)\geq N_\omega(\text{coker}(f))$. This gives inequality (a) of (4). By adding these two inequalities we get   inequality (c) of (4).  

    According to statements (2) and (3), we have $N_\omega(\text{ker}(f))=\omega(X\twoheadrightarrow \text{im}(f))$ and 
    $N_\omega(\text{coker}(f))=\omega(\text{im}(f)\hookrightarrow Y)$. These equalities together with the  composition sub-additivity gives  inequality (b) of (4).

\noindent
(5):\quad This is a consequence of inequalities (b) and (c) in statement (4).
\end{proof}
\begin{point}
For a fixed norm $N$ on $\mathcal {M}$, let us  consider the collection of all the valuations $\omega$ on $\mathcal{M}$  which are push-out, pull-back non-expansive,  composition sub-additive, and for which $N_\omega=N$ (see~\ref{adsgdsjgk}.(1)). If 
$T$ is a set of such valuations, then it is not difficult to show that 
$\text{sup}(T)$ is also such a valuation.
This means that this collection, with the 
 relation $\leq $, is an upper semi-lattice (see~\cite{sfbsfgbfgsbsffsgb}). Moreover, according to~\ref{adsgdsjgk}.(4),
 the homogeneous valuation associated with $N$ (see~\ref{adgfadfhfg}) is the global maximum in this semi-lattice. 
This semi-lattice contains also a global minimum, which is the  valuation given by the assignment $f\mapsto 
\text{\rm max}\{N(\text{\rm ker}(f)), N(\text{\rm coker}(f))\}$. Thus this 
upper semi-lattice is in fact a lattice.

More generally, assume $\alpha\colon [0,\infty)^2\to [0,\infty)$ is a function satisfying the  following properties: (1) it is symmetric: $\alpha(x_1, x_2)=\alpha(x_2,x_1)$, (2) $\alpha(x_1,0)=x_1$ for every $x_1$ in $[0,\infty)$, (3) it is non decreasing: if $x_1\leq y_1$ and $x_2\leq y_2$, then $\alpha(x_1,x_2)\leq \alpha(y_1,y_2)$, and (3) it is sub-additive: $\alpha(x_1+x_2, y_1+y_2)\leq \alpha(x_1,x_2) + \alpha(y_1,y_2)$. Then the following assignment 
is a valuation on $\mathcal{M}$ which is push-out, pull-back non-expansive,  composition sub-additive, and whose  norm is $N$ (this can be proven using exactly the same arguments as in Proposition~\ref{asgsdfh}):
\[f\mapsto \begin{cases}
    \infty & \text{ if $N(\text{ker}(f))=\infty$ or $N(\text{coker}(f))=\infty$ }\\
    \alpha(N(\text{ker}(f)), N(\text{coker}(f))) & \text{ otherwise.}
\end{cases}\]
 For example, $\alpha(x_1,x_2)=(x_1^p+x_2^p)^{1/p}$ is such a function if $p\geq 1$
(compare with~\cite{MR4776401}).
\end{point}

\begin{point}\label{dlntbuzxwq}
    A norm $N$ on $\M$ is called \textbf{ultra} if the associated homogeneous valuation $N$ is ultra (see~\ref{adfsgdfghjgukj}). In this context, note that the following statements are equivalent about a homogeneous valuation $N$ on $\M$: (1)  $N(f\oplus g) \leq \text{\rm max}\{N(f), N(g)\}$ for all monomorphisms $f$ and $g$ in $\M$,  and (1')
    $N(X\oplus Y)=\text{\rm max}\{N(X), N(Y)\}$ for all objects $X$ and $Y$ in $\M$. 
\end{point}

\begin{prop}\label{asdfdfghdfhsfgh}
    Assume  $\omega$ is a valuation on an abelian category $\M$ which is push-out,  pull-back non-expansive, and composition sub-additive. 
    \begin{enumerate}
        \item The valuation $\omega$ is ultra if and only if the norm $N_\omega$ is ultra. Moreover, for every $t$ in $[0,\infty)$, the $t$-shift with respect to $\omega$
        coincides with the $t$-shift with respect to $N_\omega$.
        \item If $\omega$ is ultra, then, for every epimorphism 
         $f\colon X\to Y$ and  every $t$ in $[0,\infty)$, the $t$-shift $f[t]\colon X[t]\to Y[t]$ is also an epimorphism.
    \end{enumerate}
\end{prop}
\begin{proof}
\noindent
(1):\quad This is a direct consequence of Proposition~\ref{adsgdsjgk}.(2).

\noindent
(2):\quad
Consider the monomorphism
 $\text{im}(f\circ \text{T}_X[t])\hookrightarrow Y$. Since its cokernel receives an epimorphism from the cokernel of $T_X[t]$, Proposition~\ref{adsgdsjgk}.(2) implies that the valuation of this monomorphism is bounded by $t$. Thus, according to~\ref{asgdsghdfgjh}, there is a unique monomorphism $h\colon Y[t]\hookrightarrow \text{im}(f\circ\text{T}_X[t])$ for which the following diagram commutes where the indicated morphisms are epimorphisms and monomorphisms:
    \[\begin{tikzcd}
        X[t]\ar{d}[swap]{f[t]}\ar[two heads]{dr}\ar[hook]{rr}{\text{T}_X[t]} &  &X\ar[two heads]{d}{f}\\
        Y[t]\ar[hook]{r}{h} \ar[bend right=30pt, hook]{rr}{\text{T}_Y[t]} &  \text{im}(f\circ\text{T}_X[t])\ar[hook]{r} & Y
    \end{tikzcd}\]
    This can happen only if $h\colon Y[t]\to  \text{im}(f\circ\text{T}_X[t])$ is an isomorphism, i.e., when $f[t]$ is an epimorphism. 
\end{proof}

\section{Simple valuations on tame representations of $[0,\infty)^\ell$}
In the previous sections we presented categorical and universal definitions without giving any examples. The aim of this  section is to exemplify 
these definitions in a concrete case of representations of the
poset of $\ell$-tuples of non-negative real numbers.
\begin{point}\label{dfgsfdhfg}
The symbol $[0,\infty)^\ell$ denotes the poset of $\ell$-tuples of non-negative real numbers equipped with the relation $(v_1,\ldots v_\ell)\geq (w_1,\ldots w_\ell)$ if $v_i\geq w_i$ for every $i$.
The symbol $[0,\infty)^\ell_\infty$ denotes the poset obtained by adding an additional element $\infty$ to  $[0,\infty)^\ell$, with this element being the global maximum of $[0,\infty)^\ell_\infty$. An element $v$ in  $[0,\infty)^\ell_\infty$ is called \textbf{finite} if $v<\infty$. For $v,w$ in $[0,\infty)^\ell_\infty$ we write  $v\gg w$ if either $v=\infty$ and $w$ is finite, or both $v, w$ are finite and $v_i> w_i$ for every $i$.

Recall that the  posets $[0,\infty)^\ell$ and $[0,\infty)^\ell_\infty$ are distributive lattices, and  we use the symbols $v\wedge w$ and $v\vee w$  to denote the meet and the join of their elements $v$ and $w$ respectively. For example, if $v$ and $w$ are finite, then
$v\wedge w= (\text{min}\{v_1,w_1\}, \ldots, \text{min}\{v_\ell,w_\ell\})$ and $v\vee w= (\text{max}\{v_1,w_1\}, \ldots, \text{max}\{v_\ell,w_\ell\})$. For every subset $T\subset [0,\infty)^\ell_\infty$, there is a unique element $\text{inf}(T)$ in $[0,\infty)^\ell_\infty$ characterized by two requirements: (1) $\text{inf}(T)\leq w$ for every $w$ in $T$, and (2) $u\leq \text{inf}(T)$ for every  element
$u$ in $[0,\infty)^\ell_\infty$ for which $u\leq w$ for every $w$ in $T$. For example, $\text{inf}(\emptyset)=\infty$.

Let  $K$ be a field.  The category of finite dimensional $K$ vector spaces is denoted by  $\text{vec}_K$. 
    Recall that the category of $\text{vec}_K$ valued functors indexed by the poset
    $[0,\infty)^\ell$, with natural transformations as morphisms, is an  Abelian category.
\end{point}
\begin{point}\label{sadgsdgfhjfg}
    For an element $v$ in $[0,\infty)^\ell_\infty$, 
    the symbol $K(v,-)\colon [0,\infty)^\ell\to \text{vec}_K$ denotes the following functor:
    \[K(v,x\leq y)=\begin{cases}
        0\to 0 &\text{ if } v\not\leq x \text{ and }v\not\leq y\\
        0\to K &\text{ if } v\not\leq x \text{ and }v\leq y\\
        K\xrightarrow{\text{id}} K &\text{ if } v\leq x
    \end{cases}\]
    For example, $K(\infty,-)$ is the zero functor. If $v<\infty$,
    then the functor $K(v,-)$ is called \textbf{free on one generator in degree $v$}, since, for any other functor $X\colon [0,\infty)^\ell\to \text{vec}_K$, the function $\text{Nat}(K(v,-),X)\to X(v)$, $\varphi\mapsto \varphi_v(1)$, is a linear isomorphism. Thus, for $v<\infty$, an element $x$ in $X(v)$ determines a unique natural transformation, denoted by the same symbol $x\colon K(v,-)\to X$,
    and all natural transformations $ K(v,-)\to X$ are of this form.
    In particular, for $v,w<\infty$:
    \[\text{Nat}(K(v,-),K(w,-))=\begin{cases}
        K &\text{ if }w\leq v\\
        0&\text{ if }w\not\leq v
    \end{cases}\]
    More generally,  for sequences of elements $(v_i)_{1\leq i\leq m}$
    and  $(w_i)_{1\leq i\leq n}$ in $[0,\infty)^\ell$,
    the vector space of natural transformations
    $\text{Nat}(\bigoplus_{j=1}^mK(v_j,-), \bigoplus_{i=1}^n K(w_i,-))$ is identified with the vector space of $n\times m$ matrices $M$ of elements in $K$ such that
    if $M_{i,j}\not=0$, then $w_i\leq v_j$. Functors of the form 
    $\bigoplus_{j=1}^mK(v_j,-)$ are called \textbf{free}.

    A sequence of elements $(x_i\in X(v_i))_{i\in I}$  is said to generate a functor $X\colon [0,\infty)^\ell\to \text{vec}_K$, if the natural transformation $[x_i \mid i\in I]\colon \bigoplus_{i\in I}K(v_i,-)\to X$ 
    is an epimorphism.
\end{point}
\begin{point}\label{sdgdfgj}
    A functor $X\colon [0,\infty)^\ell\to \text{vec}_K$
    is called \textbf{tame} if it is finitely presented, i.e., if there are functions $\beta_0,\beta_1\colon [0,\infty)^\ell\to \mathbf{N}$ such that: (1) their supports
    $\text{supp}(\beta_0):=\{v\in [0,\infty)^\ell \mid \beta_0(v)>0\}$ and $\text{supp}(\beta_1)$ are finite sets, and (2) there is an exact sequence, called a free presentation of $X$:
    \begin{equation}\begin{tikzcd}
        \displaystyle{\bigoplus_{v\in [0,\infty)^\ell}K(v,-)^{\beta_1(v)}}\ar{r} & 
        \displaystyle{\bigoplus_{v\in [0,\infty)^\ell}K(v,-)^{\beta_0(v)}}\ar{r} & X\ar{r} &0.
    \end{tikzcd}\tag{$\ast\!\ast\!\ast$}\end{equation}
    For example, for every $v$ in $[0,\infty)^\ell_\infty$, the functor $K(v,-)$ is tame. If $X$ is tame, then, for every $v$ in $[0,\infty)^\ell$, there is $\varepsilon >0$, such that, for every $v\leq w<v+(\varepsilon,\cdots,\varepsilon)$, the function $X(v\leq w)$ is an isomorphism (see for instance~\cite{sfbsfgbfgsbsffsgb, MR3735858, corbet2018representation}). 
    
    Among free presentations of a tame functor, there are minimal ones, and the associated functions are denoted by
    $\beta_0X,\beta_1X\colon [0,\infty)^\ell\to \mathbf{N}$ and called \textbf{Betti diagrams} of $X$ (see for example~\cite{LesnickMichael2022CMPa, ChacholskiWojciech2025KCaR}). The  number $|\beta_0X|\coloneqq\sum\beta_0X$  describes the minimal number of generators of $X$. The number
    $|\beta_1X|\coloneqq\sum\beta_1X$ describes the minimal number of relations of $X$. For example, for $v<\infty$ we have $\beta_0K(v,-)(w)=0$ if $v\not=w$, 
     $\beta_0K(v,-)(v)=1$, and $\beta_1K(v,-)(w)=0$ for all $w$. 

     The category of tame functors with natural transformations as morphisms is denoted by
     $\text{tame}([0,\infty)^\ell, \text{vec}_K)$. This category is Abelian~\cite{gdsfgdfgfdgg}. 
     
     By identifying  elements $v$ in 
     $[0,\infty)^\ell_\infty$ 
     with the functors 
     $K(v,-)$, we  think about $[0,\infty)^\ell_\infty$ as a subset of the set of objects of the category 
     $\text{tame}([0,\infty)^\ell, \text{vec}_K)$. Thus any pseudo-metric on the set of tame functors, by restriction, induces a pseudo-metric on the sets $[0,\infty)^\ell\subset [0,\infty)^\ell_\infty$. For instance, a valuation $\omega$ on $\text{tame}([0,\infty)^\ell, \text{vec}_K)$ leads to a pseudo-metric $d_\omega$ on $\text{tame}([0,\infty)^\ell, \text{vec}_K)$ (see~\ref{adgsdfh}) and hence also on  
     $[0,\infty)^\ell_\infty$. 
\end{point}

\begin{point}\label{dfhdfgjghkj}
A valuation $\omega$ on  $\text{tame}([0,\infty)^\ell, \text{vec}_K)$ (see~\ref{adgsdfh}) is called \textbf{simple} if
    it satisfies the following conditions:
    (0) it is  push-out and pull-back non expansive, composition sub-additive, and monotone  (see~\ref{asgdhjfdh}); 
    (1) $\omega(f\oplus g) \leq \text{\rm max}\{\omega(f),\omega(g)\}$ for all monomorphisms $f$ and $g$; (2)  the shift  $T_X[t]\colon X[t]\subset X$ with respect to $\omega$ exists for every $t$ and  $X$ (see~\ref{adfsgdfghjgukj}); and (3) the shift does not increase number of generators:
    $|\beta_0 X[t]|\leq |\beta_0 X|$ for every  $t$ and $X$. 
     
    The requirements of being pull-back non expansive  together with (1) and (2) above state that a simple valuation  is ultra (see~\ref{adfsgdfghjgukj}). 

    A norm $N$ on $\text{tame}([0,\infty)^\ell, \text{vec}_K)$ (see~\ref{adfsghdfgjh}) is called \textbf{simple} if the associated homogeneous valuation (see~\ref{adgfadfhfg}) is simple. 
    Since a homogeneous  valuation  is always 
    push-out and pull-back non-expansive, composition sub-additive, and monotone (see~\ref{asgsdfh}), a  norm $N$ is simple if, and only if, it satisfies: (1) $N(X\oplus Y)= \text{max}\{N(X),N(Y)\}$ for all tame functors  $X$ and $Y$ (see~\ref{dlntbuzxwq}); (2) the shift $X[t]\subset X$ exists for every $t$ and $X$; and (3)  $|\beta_0 X[t]|\leq |\beta_0 X|$ for every  $t$ and $X$.

    Assume $\omega$ is a valuation satisfying the requirement (0) of being simple (it is push-out,  pull-back non-expansive,  composition sub-additive, and monotone).
    Then according to~\ref{asdfdfghdfhsfgh}.(1), $\omega$ is ultra if, and only if, the  norm $N_\omega$ (see~\ref{adsgdsjgk}.(1)) is ultra, moreover their shifts agree. Thus in this case $\omega$ is simple if, and only if, the norm $N_\omega$ is simple.

\end{point}
\begin{point}\label{asdgfdfhf}
    Let $\omega$ be a simple valuation on $\text{tame}([0,\infty)^\ell, \text{vec}_K)$.  We are going to describe how
    such a valuation leads to a function $\ast_\omega\colon [0,\infty)^\ell_\infty\times [0,\infty)\to [0,\infty)^\ell_\infty$.
    For an element $v$ in $[0,\infty)^\ell_\infty$ and $t$ in $[0,\infty)$, consider the shift $K(v,-)[t]\subset K(v,-)$.
    If $v=\infty$, then $K(\infty,-)[t]=0=K(\infty,-)$ and we set 
    $\infty\ast_\omega t:=\infty$. If $v<\infty$, by requirement (3) of the simplicity,  $|\beta_0K(v,-)[t]|\leq |\beta_0K(v,-)|= 1$. Thus either $|\beta_0K(v,-)[t]|=0$, or
    $|\beta_0K(v,-)[t]|=1$. The first case happens if, and only if, $K(v,-)[t]=0=K(\infty,-)$, which is equivalent to  $\omega(0\to K(v,-))\leq t$. In this case we also set $v\ast_\omega t:=\infty$.
    In the second case, since  $K(v,-)[t]$ is generated by one element and non-zero  submodule of the free module $K(v,-)$, it has to be of the form $K(v\ast_\omega t,-)$ for some element $v\ast_\omega t\geq v$ in $[0,\infty)^\ell$. To summarize: 
    \begin{center}{\em
    $v\ast_\omega t$ is the unique element in $[0,\infty)^\ell_\infty$ for which
    $K(v,-)[t]=K(v\ast_\omega t,-)$. }
    \end{center}
    In particular,  $v\ast_\omega t=\infty$ if, and only if,   $\omega(0\to K(v,-))\leq t$.

    As the function $\ast_\omega$ depends only on the shift, and since the shifts for $\omega$ and  the  norm $N_\omega$ (se~\ref{adsgdsjgk}) agree, the function $\ast_\omega$ is determined by the norm $N_\omega$.
\end{point}

   The function $\ast_\omega$ enjoys  properties which  state in particular that it is a  lax action of the additive monoid $[0,\infty)$ on the poset $[0,\infty)^\ell_\infty$:

\begin{prop}\label{adgsdghdfg}
Let $\omega$ be a simple  valuation on $\text{\rm tame}([0,\infty)^\ell, \text{\rm vec}_K)$. 
Then the function  $\ast_\omega\colon [0,\infty)^\ell_\infty\times [0,\infty)\to [0,\infty)^\ell_\infty$ satisfies the following properties:
    \begin{enumerate}
        \item $v\leq v\ast_\omega t$ for every $v$ in  $[0,\infty)^\ell_\infty$ and 
        $t$ in $[0,\infty)$.
        \item $v\ast_\omega s\leq w\ast_\omega t$ for every pair $v\leq w$ in $[0,\infty)^\ell_\infty$ and $s\leq t$ in $[0,\infty)$.
        \item $(v\ast_\omega s)\ast_\omega t\leq v\ast_\omega(s+t)$ for every $v$ in $[0,\infty)^\ell_\infty$ and $s,t$ in $[0,\infty)$. 
        \item For every pair of elements $v$ and $w$ in
        $[0,\infty)^\ell_\infty$, the set $\{t\in [0,\infty)\ |\
        v\ast_\omega t\geq w\}$ is either empty or is of the form $[\delta(v,w),\infty)$ for some element $\delta(v,w)$ in $[0,\infty)$.
      \end{enumerate}
\end{prop}
\begin{proof}
Property (1) is clear from the definition.  Statement (2) is equivalent  to
the existence of a monomorphism $K(w,-)[t]\subset  K(v,-)[s]$ which is guaranteed by Proposition~\ref{sadgsdfhsfgh}.(2).  Similarly,  since $\omega$ is   composition sub-additive, Proposition~\ref{sadgsdfhsfgh}.(3) gives the existence of a mononomorphism  $K(v,-)[s+t]\subset (K(v,-)[s])[t]$ which is equivalent to statement (3). 

To prove property (4), note that the set $T=\{t\in [0,\infty) \mid v\ast_\omega t\geq w\}$ consists of these $t$ for which the following inclusions hold $K(v,-)[t]=K(v\ast_\omega t, -)\subset K(v\vee  w,-)\subset K(v,-)$. These inclusions are equivalent to
the inequality $\omega(K(v\vee  w,-)\subset K(v,-))\leq t$
(see~\ref{sadgsdfhsfgh}.(4)).
If we define $\delta(v,w):=\omega(K(v\vee  w,-)\subset K(v,-))$, then the considered set $T$ is either empty, if $\delta(v,w)=\infty$, or is the closed interval $[\delta(v,w),\infty)$ if $\delta(v,w)<\infty$. 
\end{proof}

Property (4) of  Proposition~\ref{adgsdghdfg}  can be 
reformulated in another way:
\begin{prop}\label{adfdsfshtf}
    Let $\ast\colon [0,\infty)^\ell_\infty\times [0,\infty)\to [0,\infty)^\ell_\infty$ be a function satisfying  property (2) of Proposition~\ref{adgsdghdfg}. Then $\ast$ satisfies  property (4) of  Proposition~\ref{adgsdghdfg} if and only if:
    \begin{itemize}
        \item[(4')] $v\ast\text{\rm inf}(T)=\text{\rm inf}(v\ast T)$ for every non-empty subset $T\subset [0,\infty)$ and   $v$ in $[0,\infty)^\ell_\infty$.
    \end{itemize}
\end{prop}
\begin{proof}
    Assume  $\ast$ satisfies  property (4) of Proposition~\ref{adgsdghdfg}. 
    Then, since it also satisfies  property (2), $v\ast\text{\rm inf}(T)\leq v\ast t$ for every $t$ in $T$, and hence $v\ast\text{\rm inf}(T)\leq \text{\rm inf}(v\ast T)$.
    Consider the set $\{t\in [0,\infty) \mid v\ast t \geq \text{\rm inf}(v\ast T)\}$. Since this set contains $T$ and is of the form $[a,\infty)$, it has to also contain $\text{inf}(T)$, which gives  $v\ast\text{\rm inf}(T)\geq \text{\rm inf}(v\ast T)$. 

    Assume  $\ast$ satisfies  property $(4')$. For $v,w$ in $[0,\infty)^\ell_\infty$, consider  $T=\{t\in [0,\infty) \mid v\ast t\geq w\}$. 
    If $T$ is non-empty, then, according to our assumption
    $v\ast\text{\rm inf}(T)=\text{\rm inf}(v\ast T)\geq w$.
    It follows that  $\text{\rm inf}(T)$ belongs to $T$, and since 
    by property (2) $T$ is an upset, it has to be of the required form $[a,\infty)$ for  property (4) of Proposition~\ref{adgsdghdfg} to hold.
 \end{proof}

\section{Contours and  step functions}\label{section6}
\begin{point}\label{assfadfhfgs}
    A function $\ast\colon [0,\infty)^\ell_\infty\times [0,\infty)\to [0,\infty)^\ell_\infty$ 
    is called  a \textbf{contour} (in dimension $\ell$) if it satisfies the four properties  (1)-(4) of Proposition~\ref{adgsdghdfg}.  Equivalently, $\ast$ is a contour if, in addition to properties (1)-(3) of Proposition~\ref{adgsdghdfg}, it  satisfies  property (4') of Proposition~\ref{adfdsfshtf}. Our definition of a contour differs from those presented in~\cite{chacholski2020metrics} and~\cite{gafvert2021stableinvariantsmultiparameterpersistence}. It additionally includes requirement (4).

    If $\omega$ is a simple   valuation on $\text{tame}([0,\infty)^\ell, \text{vec}_K)$, then  the induced function $\ast_\omega$ 
     (see~\ref{asdgfdfhf})  is an example of a contour.  However, the  function defined as follows:
     \[v\ast t:=\begin{cases} v & \text{ if } t=0\\
     \infty & \text{ if } t>0
     \end{cases}\] 
     although it satisfies the first three properties of  Proposition~\ref{adgsdghdfg}, is not a contour according to our definition, since, for every $v$ in $[0,\infty)^\ell$, the set $\{t\in [0,\infty) \mid v\ast t=\infty\}=(0,\infty)$ is not  of the form required by property (4) of~\ref{adgsdghdfg}.
     
     The function defined by the formula $v\ast t:=\infty$,
     for all $v$ in $[0,\infty)^\ell_\infty$ and $t$ in $[0,\infty)$, is a contour called \textbf{trivial}. 
     
     The function defined by the following formula is also a contour called \textbf{standard}:
     \[
     v\ast t := \begin{cases}
         \infty & \text{ if }v=\infty\\
         v+(t,\ldots, t) & \text{ if }v<\infty
     \end{cases}
     \] 

      If $v\ast 0 = v$ for every $v$ in $[0,\infty)^\ell_\infty$, then the contour $\ast$ is called \textbf{unital}. 
 \end{point}

 \begin{point}\label{asdvghnghmn}
    For a contour $\ast\colon [0,\infty)^\ell_\infty\times [0,\infty)\to [0,\infty)^\ell_\infty$,  define its \textbf{step function} $\delta\colon [0,\infty)^\ell_\infty\times [0,\infty)^\ell_\infty\to [0,\infty]$ by the following formula:
     \[\delta(v,w):=\text{inf}\{t\in [0,\infty) \mid v\ast t\geq w\}\]
     Thus the set $\{t\in [0,\infty) \mid v\ast t\geq w\}$ is non-empty if and only if $\delta(v,w)<\infty$, in which case  
     $\{t\in [0,\infty) \mid v\ast t\geq w\}=[\delta(v,w), \infty)$
     (see property (4) of Proposition~\ref{adgsdghdfg}).

    For example, for the trivial contour, $\delta(v,w)=0$ for every $v$ and $w$ in $[0,\infty)^\ell_\infty$. For the standard contour:
      \[\delta(v,w)=\begin{cases}
         0 & \text{ if } v=\infty\\
         \infty & \text{ if }v<w=\infty\\
         \text{max}\{0, w_1-v_1,\ldots w_\ell-v_\ell\} & \text{ if }v,w<\infty
     \end{cases}\]
     In the case of the contour $\ast_\omega$,   associated with  a simple  valuation $\omega$ on $\text{tame}([0,\infty)^\ell, \text{vec}_K)$ (see~\ref{asdgfdfhf}), we have  
     $\delta(v,w)=\omega(K(v\vee  w,-)\subset K(v,-))$
      (see the proof of Proposition~\ref{adgsdghdfg} and compare with~\ref{asdgsdfhdfgjh}). 

The following properties of the step function $\delta$, associated with the contour $\ast$,  follow directly from the definition and the properties of  the contour.
\begin{prop}\label{asghfsgjfghdkj}
    \begin{enumerate}
        \item If $v\leq w$, then $\delta(w,v)=0$.
        \item  If $v\leq w$, then $\delta(w,u)\leq \delta(v,u)$.
        \item $\delta(v,z)\leq \delta(v,w)+\delta(w,z)$ for every $v,w,z$ in $[0,\infty)^\ell_\infty$.
        \item For every $v$ in $[0,\infty)^\ell_\infty$ and $t$ in $[0,\infty)$,  the set $D_{\delta}(v,t)\coloneqq\{u\in [0,
        \infty)^\ell_\infty\mid \delta(v,u)\leq t\}$ is a down set of the form
        $\{u\in [0,\infty)^\ell_\infty\mid u\leq a\}$ for some element  $a$ in $[0, \infty)^\ell_\infty$ (in this case $a=v\ast t$).
    \end{enumerate}
\end{prop}
\end{point}

\begin{point}\label{sdfgsdfhf}
 Note that if $w\ast 0=w$, then the implication
 (1) in~\ref{asghfsgjfghdkj}  can be reversed since $\delta(w,v)=0$ would imply that $v\leq w\ast 0=w$. Consequently,
 if $w\ast 0=w$ and $v\ast 0=v$, then 
 $\delta(w,v)=\delta(v,w)=0$ happens if and only if $w=v$. It follows that a contour $\ast$ is unital
 (see~\ref{assfadfhfgs}) if, and only if,  the equalities $\delta(w,v)=\delta(v,w)=0$ happen only when $v=w$.

The step function $\delta$ can be symmetrized, e.g.\  by $d(v,w) = \delta(v,w) + \delta(w,v)$ (compare with~\ref{asfhgfgdjgh}), to obtain a pseudometric on $[0,\infty)^\ell_\infty$, where the triangle inequality is implied by~\ref{asghfsgjfghdkj}.(3).  Furthermore, 
 if in addition the contour is unital, then $d$ is a metric on $[0,\infty)^\ell_\infty$.
    \end{point}

\begin{point}
A function $\delta\colon [0,\infty)^\ell_\infty\times [0,\infty)^\ell_\infty\to [0,\infty]$  that 
satisfies  the  four properties (1)-(4) in Proposition~\ref{asghfsgjfghdkj} is called a \textbf{step function} (in dimension $\ell$). 

If $\delta$ is such a function, then we claim that the following function is a contour:
        \[\ast\colon [0,\infty)^\ell_\infty\times [0,\infty)\to [0,\infty)^\ell_\infty\ \ \  \ \ v\ast t\coloneqq \text{max}\,D_\delta(v,t) \text{ (see~\ref{asghfsgjfghdkj}.(4))}\]
        Property (1) of Proposition~\ref{adgsdghdfg} follows from the fact that according to \ref{asghfsgjfghdkj}.(1) every $v$ belongs to $D_{\delta}(v,t)$.
        Property (2) of~\ref{adgsdghdfg} follows since according to \ref{asghfsgjfghdkj}.(2), if $v\leq w$ and $s\leq t$, then $D_{\delta}(v,s)\subset D_{\delta}(w,s)\subset D_{\delta}(w,t)$.
        Property (3) of~\ref{adgsdghdfg} is a consequence of \ref{asghfsgjfghdkj}.(3)
        since it implies  $D_\delta(v\ast s, t)\subset D_\delta(v,s+t)$. Finally, for property (4) of~\ref{adgsdghdfg}
        note that:
        \[\{t\in [0,\infty)\mid v\ast t\geq w\}=\{t\in [0,\infty)\mid w\in D_{\delta}(v,t)\}=
        \{t\in [0,\infty)\mid \delta(v,w)\leq t\}\]
        Thus this set  is either empty if $\delta(v,w)=\infty$ or it is of the form
        $[\delta(v,w),\infty)$ if $\delta(v,w)<\infty$. 

        This discussion can be summarized by:
    \end{point}
\begin{prop}\label{sdgsdfgh}
The     function that assigns to a contour its step function is a bijection between the set of contours in dimension $\ell$ and the set of step functions in dimension $\ell$.
\end{prop}

Next we exhibit several natural constructions which produce new contours from old ones. Equivalently, these constructions define endomorphisms on the set of contours, where each operation modifies a contour in a different geometric way. These constructions provide a rich calculus of contours.

\begin{point}\label{sdgdghj}
    Let $\varphi\colon [0,\infty)^\ell_\infty\to [0,\infty)^\ell_\infty$ be an order isomorphism, i.e., a bijection such that $v\leq w$ if, and only if, $\varphi(v)\leq \varphi(w)$.
    Note that the inverse  $\varphi^{-1}\colon [0,\infty)^\ell_\infty\to [0,\infty)^\ell_\infty$ is also an order isomorphism. Moreover $\varphi(0,\ldots, 0)=(0,\ldots, 0)$ and $\varphi(\infty)=\infty$.
    
    Order isomorphisms can be used to transform contours. 
    For a contour  $\ast$ in dimension $\ell$, define
$\ast^\varphi\colon [0,\infty)^\ell_{\infty}\times [0,\infty)\to [0,\infty)^\ell_{\infty}$ as follows:
\[v\ast^\varphi t\coloneqq  \varphi^{-1}(\varphi(v)\ast t)\]
For instance, $v\ast^\varphi t=\infty$ if and only if either $v=\infty$ or $\varphi(v)\ast t=\infty$.

We claim that $\ast^\varphi$ is also a contour.  
It satisfies  
requirements (1) and (2) of Proposition~\ref{adgsdghdfg}
since  
 both $\varphi$ and $\varphi^{-1}$ are order preserving.
Requirement (3) of Proposition~\ref{adgsdghdfg} follows from:
\[
 (v\ast^\varphi t)\ast^\varphi s = \varphi^{-1}(\varphi(v)\ast t)\ast^\varphi s=
 \varphi^{-1}((\varphi(v)\ast t)\ast s)\leq
 \varphi^{-1}(\varphi(v)\ast (t+ s))=v\ast^\varphi(t+s).
 \]
 Finally, for requirement (4), note that:
\[\{t\in[0,\infty)\mid
        v\ast^\varphi t\geq w\}= 
        \{t\in[0,\infty) \mid         \varphi^{-1}(\varphi(v)\ast t)\geq w\}=\{t\in[0,\infty)\ |\
        \varphi(v)\ast t\geq \varphi(w)\}.\]
 
 This discussion also identifies the step function $\delta^\varphi$ of the contour $\ast^\varphi$ as the  composition:
 \[ \begin{tikzcd}[row sep=5pt]
    {[0,\infty)_{\infty}^\ell\times [0,\infty)^\ell_\infty} \ar{rr}{\varphi\times \varphi}
    \ar[bend right= 10pt]{rd}{\delta^\varphi}
    & & 
    {[0,\infty)_{\infty}^\ell\times [0,\infty)^\ell_\infty} 
    \ar[bend left= 10pt]{ld}[swap]{\delta}\\
    & {[0,\infty]}
    \end{tikzcd}
    \]

Note that if $\varphi,\psi\colon [0,\infty)^\ell\to [0,\infty)^\ell$ are ordered isomorphisms, then 
$(\ast^\varphi)^{\psi}=\ast^{\varphi\psi}$. Moreover,
$\ast^{\text{id}}=\ast$. Thus the assignment $\varphi\mapsto(\ast \mapsto \ast^\varphi)$ is an action of the group of order isomorphisms of $[0,\infty)^\ell$ on the set of contours. 
The stabilizer group of a contour $\ast$, with respect  to this action, consists of  order isomorphisms $\varphi\colon [0,\infty)^\ell\to [0,\infty)^\ell$ satisfying the following conditions for all $v$ in $[0,\infty)^\ell$ and $t$ in $[0,\infty)$: (1) $v\ast t=\infty$ if and only if $\varphi(v)\ast t=\infty$, and (2) if $v\ast t<\infty$, then $\varphi(v\ast t)=\varphi(v)\ast t$.
For example, all order isomorphisms stabilize  the trivial contour and hence the trivial  contour is a fixed point of this action. An order isomorphism $\varphi$ stabilizes the standard contour if and only if, for every $v$ in 
$[0,\infty)^\ell$ and $t$ in $[0,\infty)$, the following equality holds: $\varphi(v+(t,\ldots,t))=\varphi(v)+ (t,\ldots,t)$. Thus  such an order isomorphism is entirely determined by its values on the boundary $\{v\in [0,\infty)^\ell \mid \text{there is $i$ such that }v_i=0\}$ of $[0,\infty)^\ell$. For instance, in the case $l=1$, the stabilizer of the standard contour is trivial. An important consequence of this discussion is that the orbit of the standard contour is a rich source of examples of contours. 
\end{point}

\begin{point}\label{asdfdfhdghj}
Let $\Gamma\subset [0,\infty)^\ell$ be a finite subset. 
     For a contour $\ast$ in dimension $l$, define
     its \textbf{truncation} $\ast_{/\Gamma}\colon [0,\infty)^\ell_\infty\times [0,\infty)\to [0,\infty)^\ell_\infty$  as follows:
     \[v\ast_{/\Gamma} t \coloneqq\begin{cases}
     \infty &\text{ if there is $x$ in $\Gamma$ s.t. }v\ast t\geq x \\
     v\ast t&\text{ otherwise.}
     \end{cases}\]

We claim $\ast_{/\Gamma}$ is also a contour. 
Requirement (1)  of Proposition~\ref{adgsdghdfg} follows
from $v\ast_{/\Gamma} t\geq v\ast t\geq v$. 
Since  requirements (2)  and (3) can be proved in a similar way, we explicitly show only (3).
The relation $(v\ast_{/\Gamma} s)\ast_{/\Gamma} t\leq v\ast_{/\Gamma}(s+t)$ is clear if $v\ast(s+t)\geq x$ for some $x$ in $\Gamma$, since in this case the right side of the relation is $\infty$.
If $v\ast(s+t)\not\geq x$ for all $x$ in $\Gamma$, then, for all $x$ in $\Gamma$ we also have $v\ast s\not\geq x$ and $(v\ast s)\ast t \not\geq x$. The required relation therefore holds since $\ast$ is a contour. 
For requirement (4), choose elements $v$ and $w$ in $[0,\infty)^\ell_\infty$, and note that:
\[\{t\in[0,\infty)\mid v\ast_{/\Gamma} t\geq w\}=
\{t\in[0,\infty)\mid v\ast t\geq w\}\cup \bigcup_{x\in \Gamma} \{t\in[0,\infty)\mid v\ast t\geq x\}.\]
Since $\Gamma$ is finite,  the set $\{t\in[0,\infty)\mid v\ast_{/\Gamma} t\geq w\}$ is  therefore closed and hence of the required form. 
This discussion also identifies the  step function $\delta_{/\Gamma}$ of the contour $\ast_{/\Gamma}$:
\[\delta_{/\Gamma}(v,w)=\text{min}\{\delta(v,w), \delta(v,x)\mid x\in \Gamma\}.\]
\end{point}
\begin{point}\label{dsdgdfh}
    Let $x$ be an element in $[0,\infty)^\ell_\infty$.
    For a contour $\ast$ in dimension $l$, define
     its \textbf{translation} $\ast_{\vee x}\colon [0,\infty)^\ell_\infty\times [0,\infty)\to [0,\infty)^\ell_\infty$  as follows: 
     $v\ast_{\vee x}t \coloneqq(v\vee x)\ast t$. It is straightforward to check that $\ast_{\vee x}$ satisfies requirements (1), (2), and (3) of Proposition~\ref{adgsdghdfg}.
     For requirement (4), choose elements $v$ and $w$ in $[0,\infty)^\ell_\infty$, and note that:
     \[\{t\in[0,\infty) \mid v\ast_{\vee x} t\geq w\} =
     \{t\in[0,\infty)\mid (v\vee x)\ast t \geq w\}.
     \]
     Thus this set is of the required form, and 
     $\delta_{\vee x}(v,w)=\delta(v\vee x, w)$.
\end{point}

\begin{point}\label{lbacmgocbd}
A function  $f\colon[0,\infty)\to [0,\infty]$ is called \textbf{contour type} if it satisfies the following two properties:
(1) it is super-additive: $f(s+t)\geq f(s)+f(t)$ for every $s$ and $t$ in $[0,\infty)$, and (2) it is right continuous: the equality 
$\text{inf}(f(T))=f(\text{inf}(T))$ holds 
for every non-empty subset $T\subset [0,\infty)$ (compare with~\ref{adfdsfshtf}). The super-additivity implies for instance either $f(0)=0$ or $f$ is the constant function with value $\infty$.

 For a contour $\ast$ in dimension $l$, if $f\colon[0,\infty)\to [0,\infty]$ is contour type, then the following function 
$\ast_f\colon [0,\infty)^\ell_{\infty}\times [0,\infty)\to [0,\infty)^\ell_{\infty}$ is also a contour, which we will call \textbf{time-reparameterization}  of $\ast$:
\[
v\ast_f t\coloneqq\begin{cases}
    v\ast f(t) &\text{ if } f(t)<\infty\\
    \infty & \text{ if } f(t)=\infty.
\end{cases}
\]
The conditions (1), (2), and (3) (see Proposition~\ref{adgsdghdfg}) follow from non-negativity and super-additivity of $f$ and condition (4') (see Proposition~\ref{adfdsfshtf}) follows from right continuity of $f$.

If $f(t) < \infty$, then $v \ast f(t) \geq w$ is equivalent to $f(t) \geq \delta(v,w)$.  Since this last inequality also happens if  $f(t) = \infty$, we have 
$\{t\in[0,\infty)\mid v\ast_f t\geq w\} = \{t \in [0, \infty)\mid f(t) \geq \delta(v,w) \}$, and hence $\delta_{\ast_f}(v,w) = \inf \{t \in [0, \infty)\mid f(t) \geq \delta(v,w) \}$.
\end{point}

\begin{point}\label{mgdfdokobh}
Another rich family of contours may be obtained by applying different one-dimensional contours (see \cite{chacholski2020metrics}) coordinatewise. On the other hand, an example of a contour $\ast\colon [0,\infty)^2_\infty\times [0,\infty)\to [0,\infty)^2_\infty$ where the coordinates do not evolve independently is given by:

\[
     v\ast t := \begin{cases}
         \infty & \text{ if }v=\infty\\
         \left(v_1 + t, v_2 + v_1t + \displaystyle\frac{t^2}{2}\right) & \text{ if }v<\infty.
     \end{cases}
     \] 
     
\end{point}

\section{Contours, simple valuations, and simple norms}
Let $\ast\colon [0,\infty)^\ell_\infty\times [0,\infty)\to [0,\infty)^\ell_\infty$ be a contour.
\begin{point}\label{sadgsfgj}
     For a natural transformation $f\colon X\to Y$, between tame functors
    $X,Y\colon [0,\infty)^\ell\to \text{vec}_K$, define the following  subset of $[0,\infty)$:
    \[c_f:=\left\{t\in [0,\infty) \;\Biggm|\;
    \begin{array}{c}
    \text{if $v$ in $[0,\infty)^\ell$ is such that $v\ast t<\infty$, then}\\
    \text{there is a   commutative diagram of the form:
    } 
    \end{array}
    \begin{tikzcd}[column sep=40pt]
        X(v)\ar{d}[swap]{f_v}\ar{r}{X(v\leq v\ast t)} &
        X(v\ast t)\ar{d}{f_{v\ast t}}\\
        Y(v)\ar{r}[swap]{Y(v\leq v\ast t)}\ar{ur} &
        Y(v\ast t)
    \end{tikzcd}
    \right\}.\]
    We refer to the maps appearing in the center of the diagrams in the definition of $c_f$ as \textbf{lifts} at $t$. 
    We claim that, if non empty, the set $c_f$ is a closed interval of the  form $[a,\infty)$ for some element $a$ in $[0,\infty)$.
    To prove this, first note that $c_f$ is an upset: if $s\leq t$ and $s$ belongs to $c_f$, then  $v\ast s\leq v\ast t$, and hence, the relation $v\ast t<\infty$ implies  $v\ast s<\infty$, giving:
    \[X(v\leq v\ast t)=X(v\ast s\leq v\ast t)X(v\leq v\ast s)=\overbrace{X(v\ast s\leq v\ast t)\text{ (lift at $s$) }}^{\text{lift at $t$}} f_v\] and
    \[
\begin{aligned}
Y(v\leq v\ast t)
&=
Y(v\ast s\leq v\ast t)\,Y(v\leq v\ast s)= \\
&=
Y(v\ast s\leq v\ast t)\,f_{v\ast s}\text{(lift at $s$)}= \\
&=
f_{v\ast t}\,
\underbrace{
X(v\ast s\leq v\ast t)\text{(lift at $s$)}
}_{\text{lift at $t$}}.
\end{aligned}
\]
    
    It remains to show that, if $c_f$ is non-empty, then $\text{inf}(c_f)$ belongs to $c_f$. Let $v$ in $[0,\infty)^\ell$ be such that $v\ast \text{inf}(c_f)<\infty$. According to Proposition~\ref{adfdsfshtf}, we have  
     $v\ast \text{inf}(c_f)=\text{inf}(v\ast c_f)$. Thus, since $X$ and $Y$ are tame, there is $t$ in $c_f$ for which the maps
     $X(v\ast \text{inf}(c_f)\leq v\ast t)$ and 
     $Y(v\ast \text{inf}(c_f)\leq v\ast t)$ are isomorphisms. The desired lift at 
     $\text{inf}(c_f)$ is then given by the composition of the lift at $t$ and the inverse $X(v\ast \text{inf}(c_f)\leq v\ast t)^{-1}$.

     Define:
    \[\omega_\ast(f):= \text{inf}(c_f).\]
    Thus,
    $\omega_\ast(f)=\infty$ if and only if $c_f$ is empty, and if  $c_f$ is non-empty, then  $c_f=[\omega_\ast(f),\infty)$.

    If $f$ is an isomorphism, then its inverse provide lifts at every $t\geq 0$. Thus in this case $c_f=[0,\infty)$ and $\omega_\ast(f)=0$.
   For the zero map $0\colon X\to Y$, if there is a lift at $t$, then the zero functions  are also its lifts at $t$. It follows that $c_{(0\colon X\to Y)}=
   c_{(0\colon Y\to X)}$ and consequently $\omega_\ast(0\colon X\to Y)= \omega_\ast(0\colon Y\to X)$. The assignment $f\mapsto \omega_\ast(f)$
   is therefore a valuation on $\text{tame}([0,\infty)^\ell, \text{vec}_K)$ (see~\ref{adgsdfh}).
\end{point}
\begin{point}\label{sdgsdghfhj}
For a tame functor $X\colon [0,\infty)^\ell\to \text{vec}_K$ 
\[ c_{X\to 0}=c_{0\to X}=\left\{t\in [0,\infty) \mid
    \text{if $v$ in $[0,\infty)^\ell$ is such that $v\ast t<\infty$, then $X(v\leq v\ast t)=0$}
    \right\}.\]
    This set is also denoted  simply by  $c_X$, and its infimum  by $N_\ast(X)$.  Thus, in this case  $N_\ast(X)=\omega_\ast(X\to 0)=\omega_\ast(0\to X)$, and
    $N_\ast(X)=\infty$ if and only if $c_X$ is empty, and if  $c_X$ is non-empty, then  $c_X=[N_\ast(X),\infty)$. This characterization gives  $c_X\cap c_Y=c_{X\oplus Y}$, for  tame functors $X$ and $Y$, implying  $N_\ast(X\oplus Y)=\text{max}\{N_\ast(X), N_\ast(Y)\}$.

    More generally, if $f$ is a monomorphism, then:
\[ c_{f}=\left\{t\in [0,\infty) \mid
    \text{if $v$ in $[0,\infty)^\ell$ is such that $v\ast t<\infty$, then $\text{coker}(f)(v\leq v\ast t)=0$}
    \right\}.\]    
     In particular, if $f$ is a monomorphism, then
    $\omega_\ast(f)=N_\ast(\text{coker}(f))$.
    Similarly, if $f$ is an epimorphism, then $\omega_\ast(f)=N_\ast(\text{ker}(f))$, and:
    \[ c_{f}=\left\{t\in [0,\infty) \mid
    \text{if $v$ in $[0,\infty)^\ell$ is such that $v\ast t<\infty$, then $\text{ker}(f)(v\leq v\ast t)=0$}
    \right\}.\]    
    
   Since taking  both cokernels and kernels commute with direct sums,
    we  also  obtain $\omega_\ast(f\oplus g)=\text{max}\{\omega_\ast(f), \omega_\ast(g)\}$ if  $f$ and $g$ are both either  epimorphisms or monomorphisms. 
    \end{point}
    \begin{prop}\label{asdgsdfhdfgjh}
    Let $\delta$ be the step function associated with the contour $\ast$ (see~\ref{asdvghnghmn}), and $u$, $w$ be elements in $[0,\infty)^\ell_\infty$. Then  $\delta(u,w)=\omega_\ast(K(u\vee  w,-)\subset K(u,-))$. 
    \end{prop}
    \begin{proof}
Let $X=\text{coker}(K(u\vee  w,-)\subset K(u,-))$.  Recall that
    $c_{(K(u\vee  w,-)\subset K(u,-))}=c_X$.
      The functor $X$ has the property that $X(v\leq v')=0$ if either $u\not\leq v$ or $v'\geq u\vee  w$. Thus in this case: 
    \[c_X=\left\{t\in [0,\infty) \mid 
    \text{if $v$ in $[0,\infty)^\ell$ is such that $v\ast t<\infty$, then either $u\not\leq v$ or $v\ast t\geq u\vee  w$} \right\}.\]
    
    Let  $t$  be such that $u\ast t\geq w$. If  $u\ast t=\infty$, then, 
    for  $v$ for which $v\ast t<\infty$, we must  have  $u\not\leq v$ and hence the implication characterizing elements of $c_X$ is satisfied, and $t$ belongs to $c_X$. If $u\ast t<\infty$, then, for  $v$ for which $v\ast t<\infty$, either $u\not\leq v$, in which case $t$ belongs to $c_X$, or $u\leq v$ and consequently $v\ast t\geq u\ast t\geq u\vee w$, which also implies $t$ is in $c_X$. This gives the inclusion  $\{t\in[0,\infty) \mid u\ast t\geq w\}\subset c_X$. On the other hand, assume $t$ belongs to $c_X$. Then, if  $u\ast t<\infty$,  we must have $u\ast t\geq u\vee w\geq w$.  If $u\ast t=\infty$, then we also have 
    $u\ast t\geq  w$. We can therefore conclude:
    \[c_X= \{t\in[0,\infty) \mid u\ast t\geq w\}.\] 
    The proposition  follows from the definition of the step function 
    (see~\ref{asdvghnghmn}).
    \end{proof}

For instance, according to ~\ref{asdgsdfhdfgjh},  for any $u$ in $[0,\infty)^\ell_\infty$, we have
    $c_{K(u, -)}=\{t\in[0,\infty) \mid u\ast t= \infty\}$.

    \begin{point}\label{sdgdfhfg}
     For a tame functor 
     $X\colon [0,\infty)^\ell\to \text{vec}_K$, let 
     us choose a  sequence   $\{x_i\in X(v_i)\}_{1\leq i\leq n}$ of its generators. 
        For $t$ in $[0,\infty)$, define $X[t]\subset X$ to be the subfunctor generated by the elements $\{X(v_i\leq v_i\ast t)(x_i) \mid \text{ for }1\leq i\leq n\text{ if } v_i\ast t<\infty\}$. The subfunctor $X[t]$ has the following property: if $v$ in $[0,\infty)^\ell$ is such that $v\ast t<\infty$,
    then $(X/X[t])(v\leq v\ast t)=0$. To see this note that
     if $v\ast t<\infty$, then $v_i\leq v$ can only happen when $v_i\ast t<\infty$. Thus every non zero element in 
    $(X/X[t])(v)$ is in the subfunctor generated by the images of the elements $\{x_i \mid 1\leq i\leq n \text{ if }v_i\ast t<\infty\}$. Consequently, $(X/X[t])(v\leq v\ast t)=0$. We can therefore conclude that $t$ belongs to
    $c_{X/X[t]}$ and hence $\omega_\ast(X[t]\subset X)=N_\ast(X/X[t])\leq t$. 

We claim that $X[t]\subset X$ satisfies the following universal property: it is the minimal, with respect to the inclusion relation, subfunctor of $X$  among all subfunctors  $Y\subset X$
     for which $\omega_\ast(Y\subset X)=N_\ast(X/Y)\leq t$.
If $Y\subset X$ is a subfunctor such that $N_\ast(X/Y)\leq t$, then,  $t$ belongs to $c_{X/Y}=[N_\ast(X/Y),\infty)$, implying that, for every $x_i$, if     
     $v_i\ast t<\infty$, then $X(v_i\leq v_i\ast t)(x_i)$ has to belong to $Y$. This means $X[t]\subset Y$.

    Thus $X[t]\subset X$ does not depend on the choice of  generators of $X$. 
     For example, for every $w$ in $[0,\infty)^\ell$ and $t$ in $[0,\infty)$, we have $K(w,-)[t]= K(w\ast t,-)$.

     Here is another consequence of this discussion:
     \end{point}
     \begin{coro}\label{sdgdghhdj}
        Let  $X\colon [0,\infty)^\ell\to \text{\rm vec}_K$ be a tame functor and $t$ an element in $[0,\infty)$. Then the following statements are equivalent:
        \begin{itemize}
         \item $N_\ast(X)\leq t$.
         \item $X[t]=0$.
         \item There is a  sequence   $\{x_i\in X(v_i)\}_{1\leq i\leq n}$ of generators of $X$ such that $X(v_i\leq v_i\ast t)(x_i)=0$ for every $i$ for which  $v_i\ast t<\infty$.
         \item For every  sequence   $\{x_i\in X(v_i)\}_{1\leq i\leq n}$ of generators of $X$, the equality $X(v_i\leq v_i\ast t)(x_i)=0$ happens for every $i$ for which  $v_i\ast t<\infty$.
     \end{itemize}
     \end{coro}

\begin{prop}\label{sadxcbgfgjh}
    The valuation $\omega_{\ast}$ (see~\ref{sadgsfgj}) is simple  (see~\ref{dfhdfgjghkj}).
\end{prop}
\begin{proof}
Consider a commutative square of tame functors:
\[\begin{tikzcd}[column sep=small, row sep=small]
           X_0\ar{r}{f}\ar{d} & X_1\ar{d}\\
           Y_0\ar{r}{g} & Y_1.
\end{tikzcd}\]
If this square is a push-out, then since 
push-outs  are formed parameter-wise, by their universal property, we have the inclusion $c_f\subset c_g$ and consequently $\omega_\ast(f)=\text{inf}(c_f)\geq \text{inf}(c_g)=\omega_\ast(g)$. Similar argument can be used to show  $\omega_\ast(g)\geq \omega_\ast(f)$ in the case this square is a pull-back. Thus the valuation $\omega_\ast$ is {\em push-out and pull-back non-expansive}.

Consider next two composable natural  transformations between tame functors $f\colon X\to Y$ and $g\colon Y\to Z$. We claim that $c_f + c_g\subset c_{gf}$, where we use the convention  $\emptyset + T=\emptyset$ for any subset $T\subset [0,\infty)$.  The inclusion $c_f + c_g\subset c_{gf}$ is clear if either $c_f$ or $c_g$ is empty. Assume that both 
    $c_f$ and $c_g$ are not empty. For $t$ in $c_f$ and $s$ in $c_g$ consider the following commutative diagram where all the horisontal arrows are the transition functions of the respective functors:
    \[\begin{tikzcd}
        X(v)\ar{r}\ar{d}{f_v} & X(v\ast s)\ar{d}[swap]{f_{v\ast s}} \ar{r} &X((v\ast s)\ast t)\ar{r}\ar{d}{f_{(v\ast s)\ast t}} & X(v\ast (s+t))\ar{d}{f_{v\ast (s+t)}} \\
        Y(v)\ar{r}\ar{d}[swap]{g_v}& Y(v\ast s)\ar{r} \ar{d}{g_{v\ast s}}\ar{ru}[description]{\text{lift at $t$}}
        &Y((v\ast s)\ast t)\ar{r}\ar{d}{g_{(v\ast s)\ast t}} & Y(v\ast (s+t))\ar{d}{g_{v\ast (s+t)}}  \\
        Z(v)\ar{r}\ar{ru}[description]{\text{lift at $s$}}\ar[bend left=45pt, dotted]{rrruu}[pos=0.7]{\text{lift at $s+t$}}& Z(v\ast s)\ar{r} &Z((v\ast s)\ast t)\ar{r}& Z(v\ast (s+t)).
    \end{tikzcd}\]
    The constructed lift at $s+t$ shows that $s+t$ belongs to $c_{gf}$, and  consequently  $\omega_{\ast}$
    is {\em composition sub-additive}:  $\omega_{\ast}(gf)\leq \omega_\ast(g) +  \omega_\ast(f)$.

To show {\em monotonicity} of $\omega_\ast$ (see~\ref{asgdhjfdh}), we need to prove  $\omega_{\ast}(gf)\geq \text{max}\{\omega_\ast(f),\omega_\ast(g)\}$  if either $f$ is an epimorphism or $g$ is a monomorphism. This is clear if $\omega_{\ast}(gf)=\infty$. Let $\omega_{\ast}(gf)<\infty$ and choose
$t$ in $c_{gf}$. 
Assume $f$ is an epimorphism. For every $v$ in $[0,\infty)^\ell$ such that $v\ast t<\infty$, we have a commutative diagram:
\[
\begin{tikzcd}[column sep=90pt]
X(v)\ar{r}\ar[two heads]{d}[swap]{f_v} & X(v\ast t)\ar[two heads]{d}{f_{v\ast t}}\\
Y(v)\ar{r}\ar{d}[swap]{g_v} & Y(v\ast t)\ar{d}{g_{v\ast t}}\\
Z(v)\ar{r}\ar[bend left=15 pt, dotted]{ruu}[description, pos=0.7]{\text{lift at $t$}} & Z(v\ast t).
\end{tikzcd}
\]
First, we claim that the composition $(\text{lift at }t )g_v$ is a lift for $f$ at $t$. The equality $(\text{lift at }t )g_vf_v= X(v\leq v\ast t)$ is clear as $t$ belongs to $c_{gf}$. Since $f_v$ is an epimorphism,  $f_{v\ast t}(\text{lift at }t )g_v = Y(v\leq v\ast t)$ happens if, and only if,:
\[f_{v\ast t}(\text{lift at }t )g_vf_v = Y(v\leq v\ast t)f_v\]
This is again  a consequence of $t$ belonging to
$c_{gf}$ as it implies $(\text{lift at }t )g_vf_v=X(v\leq v\ast t)$.  

An analogous argument can be used to prove that 
$f_{v\ast t}(\text{lift at }t )$ is a lift for $g$ at $t$.

Entirely dual arguments can be also used in the case $g$ is a monomorphism.  
The valuation $\omega_\ast$  satisfies therefore requirement (0) for being simple  (see~\ref{dfhdfgjghkj}).

The requirement (1) for  the valuation $\omega_\ast$ to be simple (see~\ref{dfhdfgjghkj}) has been already discussed at the end of~\ref{sdgsdghfhj}.

According to the discussion in~\ref{sdgdfhfg},  the natural transformation 
$X[t]\hookrightarrow X$ presented there is the  $t$-shift of a tame functor $X$ with respect to the valuation $\omega_\ast$. Finally, observe  that, by the construction, $|\beta_0X[t]|\leq |\beta_0X|$. This gives requirements (2) and (3)  for  the valuation $\omega_\ast$ to be simple (see~\ref{dfhdfgjghkj}).
\end{proof}

Since $\omega_\ast$ is simple, then so is the norm $N_\ast$ (see the end of~\ref{dfhdfgjghkj}). 

\begin{oliverthm}\label{sadgfghjg} 
    The functions $\ast\mapsto N_\ast$ and $N\mapsto \ast_N$ are inverse bijections between the set of contours in dimension $\ell$ and the set of simple norms on $\text{\rm tame}([0,\infty)^\ell, \text{\rm vec}_K)$. 
\end{oliverthm}
This theorem should be compared with~\cite[Theorem 9.6]{gafvert2021stableinvariantsmultiparameterpersistence} where  it was a mistake not to assume the 
right-continuity of contours, property (4) in Proposition~\ref{adgsdghdfg}. 
\begin{proof}
    Let $N_\ast$ be the norm on $\text{\rm tame}([0,\infty)^\ell, \text{\rm vec}_K)$ associated with a contour $\ast$ in dimension $\ell$. 
    Recall, that $K(v\ast t,-)\hookrightarrow K(v,-)$ is the  $t$-shift with respect to $N_\ast$. This means that 
    $v\ast_{N_\ast} t=v\ast t$.

    Consider  the contour $\ast_N\colon [0,\infty)^\ell_\infty\times [0,\infty)\to [0,\infty)^\ell_\infty$ associated with a simple norm $N$ on $\text{\rm tame}([0,\infty)^\ell, \text{\rm vec}_K)$ (see~\ref{asdgfdfhf}). 
    Let $X\colon [0,\infty)^\ell\to\text{vec}_K$ be a tame functor.
    By definition, the $t$-shift $X[t]$ with respect to $N$ is trivial if, and only if, $N(X)\leq t$.  Here is another characterization of $X[t]$ being trivial. 
    Let $(x_i\in X(v_i))_{i\in I}$ be a finite  sequence of generators of $X$ and $\varphi\colon \bigoplus_{i\in I}K(v_i,-)\to X$ be the induced epimorphism. 
    According to Proposition~\ref{sadgsdfhsfgh}.(1) and the discussion in~\ref{asdgfdfhf}, the $t$-shift of the functor  $\bigoplus_{i\in I}K(v_i,-)$ with respect to $N$
    can be identified with:
    \[\left(\bigoplus_{i\in I}K(v_i,-)\right)[t]=\bigoplus_{i\in I}K(v_i,-)[t]= \bigoplus_{\substack{i\in I\\ v_i\ast_N t<\infty}}K(v_i\ast_N t,-).\]
    Since according to Proposition~\ref{asdfdfghdfhsfgh}.(2), 
    the induced natural transformation $\phi[t]$ is an epimorphism, we get that $X[t]$ is trivial if and only if the following composition is trivial:
    \[\begin{tikzcd}
        \displaystyle{\bigoplus_{\substack{i\in I\\ v_i\ast_N t<\infty}}K(v_i\ast_N t,-)} \ar[hook]{r} & \displaystyle{\bigoplus_{i\in I}K(v_i,-)}\ar{r}{\varphi} & X.
    \end{tikzcd}
    \]
    This happens if, and only if, $t$ belongs to the set:
    \[\left\{t\in [0,\infty) \mid 
    \text{if $v$ in $[0,\infty)^\ell$ is such that $v\ast_N t<\infty$, then $X(v\leq v\ast_N t)=0$} 
    \right\},\]
    which is equivalent to $N_{\ast_N}(X)\leq t$. We can conclude  $N(X)\leq t$ if and only if $N_{\ast_N}(X)\leq t$. Consequently  $N=N_{\ast_N}$.
\end{proof}

\begin{point}
We finish this section with describing how the norm changes with respect to the contour  transformations discussed in sections~\ref{sdgdghj}-\ref{lbacmgocbd}. Let $\ast$ be a contour in dimension $\ell$, $\delta$ its step function, $Y\colon [0,\infty)^\ell\to \text{vec}_K$ an arbitrary tame functor,  and 
$X=\text{coker}(K(u,-)\subset K(w,-))$ for some elements  $w\leq u$  in $[0,\infty)^\ell_\infty$. Recall that $N_\ast(X)=\delta(w,u)$ (see~\ref{asdgsdfhdfgjh}).  Then:
    \begin{enumerate}
        \item For an order isomorphism $\varphi\colon [0,\infty)^\ell_\infty\to [0,\infty)^\ell_\infty$ (see~\ref{sdgdghj}), 
        $N_{\ast^{\varphi}}(Y)=N_\ast(Y\varphi^{-1})$. In particular,  $N_{\ast^\varphi}(X)=\delta(\varphi(w),\varphi_(u))$.
        \item Truncation (see~\ref{asdfdfhdghj}) decreases the norm: $N_{\ast_{/\Gamma}}(Y)\leq N_\ast(Y)$.  For the functor $X$, we can be more explicit: 
        $N_{\ast_{/\Gamma}}(X)=\delta_{/\Gamma}(w,u)=\text{min}\{N_\ast(X), \delta(w,x)\mid x\in \Gamma\}$.
        \item Translation (see~\ref{dsdgdfh}) decreases the norm as well: $N_{\ast_{\vee x}}(Y)\leq N_\ast(Y)$. As before, for the functor $X$,
        we can be more explicit: $N_{\ast_{\vee x}}(X)=\delta_\ast(w\vee x,u\vee x)$.
        \item Time reparameterization (see~\ref{lbacmgocbd}) rescales  the norms:
        $N_{\ast_f}(Y) = \text{inf}\{t\in[0,\infty)\mid f(t)\geq N_\ast(Y)\}$.
    \end{enumerate}
\end{point}

\section{Distances induced by simple norms}\label{asdfdfh}
Through out this and the next two sections we fix a simple norm $N$ on  
 $\text{\rm tame}([0,\infty)^\ell, \text{\rm vec}_K)$, a contour $\ast\colon [0,\infty)^\ell_\infty\times [0,\infty)\to [0,\infty)^\ell_\infty$, and a step function 
 $\delta\colon [0,\infty)^\ell_\infty\times [0,\infty)^\ell_\infty\to [0,\infty]$ such that
  $\ast = \ast_N$, $N=N_\ast$ (see G\"afvert's Theorem~\ref{sadgfghjg}),
  and:
  \[\delta(v,w)=\text{\rm inf}\{t\in [0,\infty) \mid v\ast t\geq w\}=N(K(v\vee w,-)\subset K(v,-))\ \ \text{(see~\ref{asdvghnghmn}, ~\ref{sdgsdfgh}, and~\ref{asdgsdfhdfgjh}).}
  \]
  
  Recall  that the norm $N$ leads to a distance $d_N$ on tame functors (see~\ref{adgsdfh}), and hence by restriction to a distance on $[0,\infty)^\ell_\infty$, denoted also by $d_N$. 
    \begin{prop}\label{asfhgfgdjgh}
     For $v$ and $w$ in $[0,\infty)^\ell_\infty$ and positive integers $n\geq m>0$:
     \[d_N(K(v,-)^{\oplus n}, K(w,-)^{\oplus m})=\begin{cases}
    \delta(v,w)+\delta(w,v), &\text{ if } n=m\\
    \delta(v,\infty)+\delta(w,v), &\text{ if } n>m.
\end{cases}\]

 \end{prop}

 Note a particular case of Proposition~\ref{asfhgfgdjgh} when $n=m=1$, which gives the following equality for all $v$ and $w$ in $[0,\infty)^\ell_\infty$:
  \[d_N(v,w)=\delta(v,w) + \delta(w,v).\]

  To prove Proposition~\ref{asfhgfgdjgh}, we start with:
 \begin{lem}\label{asfgdfhgf}
    Consider a zig-zag of tame functors in $\text{\rm tame}([0,\infty)^\ell, \text{\rm vec}_K)$:
    \[\begin{tikzcd}X& W\ar{l}[swap]{f}\ar{r}{g} & Y.
    \end{tikzcd}\]
    Assume $N(\text{\rm coker}(f))<\infty$, $N(\text{\rm ker}(g))<\infty$, and 
      $X(v\leq w)$ is a monomorphism for every pair $v\leq w$ in $[0,\infty)^\ell$. Then, for every $a$ in $\text{\rm supp}(\beta^0X)=\{v\in[0,\infty)^\ell \mid  \beta^0X(v)>0\}$ (see~\ref{sdgdfgj})
      such that $a\ast \bigl(N(\text{\rm coker}(f))+N(\text{\rm ker}(g))\bigr)<\infty $:
      \begin{enumerate}
        \item there is $b$ in $\text{\rm supp}(\beta^0Y)$   for which $a\ast N(\text{\rm coker}(f))\geq b$;
          \item $\text{\rm dim}\, X(a)\leq \text{\rm dim}\, Y(a\ast N(\text{\rm coker}(f)))$.
      \end{enumerate}
\end{lem}
\begin{proof}[Proof of Lemma~\ref{asfgdfhgf}]
Set  $a':=a\ast N(\text{\rm coker}(f))$ and
$a'':=a'\ast N(\text{\rm ker}(g))$. The  assumption  $a\ast \bigl(N(\text{\rm coker}(f))+N(\text{\rm ker}(g))<\infty $ implies  $a'\leq a''<\infty$. Consequently, for a non-zero $x$ in $X(a)$, we have $\text{\rm coker}(f)(a\leq a')([x])=0$. The element
$X(a\leq a')(x)$   is therefore of the form $f(x')$ for some $x'$ in $W(a')$.
This element $x'$ can not belong to $\text{ker}(g)$, otherwise:
$W(a'\leq a'')(x')=
\text{ker}(g)(a'\leq a'')(x')=0$, which would imply 
$X(a\leq a'')(x)=0$, contradicting the assumptions that $X(a\leq a'')$ is a monomorphism. 
Since $g(x')\not=0$, there is $b$ in $\text{\rm supp}(\beta^0Y)$ for which
$a'\geq b$ proving (1).  

Note also that this discussion leads to the following commutative diagram of vector spaces where the indicated arrows are epimorphisms and monomorphisms:
\[\begin{tikzcd}
    P\ar[two heads]{d}\ar[hook]{r}\ar[bend left=30pt, hook] {rr}& W(a')\ar{r}{g_{a'}}\ar{d}{f_{a'}} & Y(a')\\
    X(a)\ar[hook]{r} &  X(a')
\end{tikzcd}\]
Consequently we obtain (2).
\end{proof}

\begin{proof}[Proof of Proposition~\ref{asfhgfgdjgh}]
The proposition is clear in the case $v=w=\infty$.  Assume  $v<\infty$ and consider an arbitrary zig-zag:
 \[\begin{tikzcd}K(v,-)^n& W\ar{r}{g}\ar{l}[swap]{f}  & K(w,-)^m.
    \end{tikzcd}\]
\noindent
Let $q\coloneqq \begin{cases}
    \delta(v,w)+\delta(w,v), &\text{ if } n=m\\
    \delta(v,\infty)+\delta(w,v), &\text{ if } n>m.
\end{cases}$

To show $d_N(K(v,-)^n,K(w,-)^m)\geq q$, we need to prove  the following inequality:
    \[N(\text{ker}(f))+N(\text{coker}(f))+N(\text{ker}(g))+N(\text{coker}(g))\geq  q.\]
This inequality is clear if  $N(\text{ker}(f))$, or  $N(\text{coker}(f))$, or $N(\text{ker}(g))$, or $N(\text{coker}(g))$ is $\infty$. Assume that all these norms are finite. 

If $w=\infty$, then $W=\text{ker}(g)$. Thus $N(\text{coker}(f))+N(\text{ker}(g))\geq N(K(v,-)^n)=\delta(v,\infty) $, and  
since $\delta(\infty, v)=0$, we get the desired inequality.

If $w<\infty$, then we can apply Lemma~\ref{asfgdfhgf} to obtain  relations:
\[v\ast\bigl(N(\text{coker}(f))+N(\text{ker}(g))\bigr)\geq w\ \ \ \ \text{ and } \ \ \ \ w\ast\bigl(N(\text{coker}(g))+N(\text{ker}(f))\bigr)\geq v.\]
    These relations  imply the following inequalities:
    \[N(\text{coker}(f))+N(\text{ker}(g))\geq \delta(v,w)\ \ \ \ \text{ and } \ \ \ \ N(\text{coker}(g))+N(\text{ker}(f))\geq \delta(w,v).\] 
    By adding them, we obtain the desired inequality for $n=m$.

   If $n > m$ and $v\ast\bigl(N(\text{coker}(f))+N(\text{ker}(g))\bigr) = a <\infty$, then, by Lemma~\ref{asfgdfhgf} (2), we have $n= \text{\rm dim}\, K(v,-)^n(v)\leq \text{\rm dim}\, K(w,-)^m(v\ast N(\text{\rm coker}(f)))\leq m$, which contradicts the assumption.
Thus in this case $a=\infty$ and consequently $N(\text{coker}(f))+N(\text{ker}(g))\geq \delta(v,\infty)$.

       To show $d_N(K(v,-)^n,K(w,-)^m)\leq q$, consider the following zig-zag, where the inclusions are given by the diagonal matrices with entries being either $1$ or $0$ (see~\ref{sadgsdgfhjfg}):

    \[\begin{tikzcd}K(v,-)^n & K(v\vee w,-)^m\ar[hook']{l}
    \ar[hook]{r}& K(w,-)^m.
    \end{tikzcd}\]
    We have:
    \[d_N(K(v,-)^n,K(w,-)^m)\leq N\bigl((K(v\vee w,-)^m\subset K(v,-)^n\bigr)\  +\  N\bigl(K(v\vee w,-)^m\subset K(w,-)^m\bigr).\]
    Recall that   $N\bigl(K(v\vee w,-)^m\subset K(w,-)^m\bigr)=\delta(w,v)$. Similarly $N\bigl(K(v\vee w,-)^m\subset K(v,-)^n\bigr)$ equals either $\delta(v,w)$ if $n=m$ or $\delta(v, \infty)$ if $n > m$, proving the desired inequality.
\end{proof}

  Recall that the contour $\ast$ is unital ($v\ast 0=v$ for any $v$, see~\ref{assfadfhfgs}) if, and only if,  $\delta(v,w)=\delta(w,v)=0$ happens only when $v=w$ (see~\ref{sdfgsdfhf}). Thus according to~\ref{asfhgfgdjgh}, $d_N$ is an extended  metric on  $[0,\infty)^\ell_{\infty}$ if, and only if, $\ast$ is unital. It turns out that in this case  $d_N$ is also a metric on the set of   isomorphism classes of tame functors:
  
    \begin{prop}\label{asdgadsfth}
  Assume $\ast$ is unital (\ref{assfadfhfgs}). Then 
  $d_N(X,Y)=0$ if, and only if, $X$ and $Y$ are isomorphic tame functors.
\end{prop}
\begin{proof}
Assume $d_N(X,Y)=0$. Let $\alpha\colon P\subset [0,\infty)^\ell$ be any finite sub-lattice containing all the four supports 
$\text{\rm supp}(\beta_0X)$, $\text{\rm supp}(\beta_1X)$, $\text{\rm supp}(\beta_0Y)$, and $\text{\rm supp}(\beta_1Y)$. 

Note that since $\ast$ is unital, the extended numbers 
$\delta(v,\infty)$ and $\delta(v,w)$   are strictly positive for all $v$  and 
$w\nleq v$ in $[0,\infty)^\ell$.  
Thus, since $P$ is finite, there is 
a positive $\eta>0$ such that: 
 (i) $2\eta <\delta(p,\infty)$ for every $p$ in $P$, and 
(ii)  $2\eta <\delta(p,q)$ for every $q\nleq p$ in $P$. 
These  conditions are equivalent to: 
 (i) $p\ast (2\eta)<\infty$ for all $p$ in $P$, and (ii)  for $p$ and $q$ in $P$, the relation $q\leq p$ holds if, and only if, the relation  $ q\leq p\ast (2\eta)$ holds. 
 
 Consider the restrictions $X|_{P}$ and $Y|_{P}$ and their left Kan extensions $\alpha^k(X|_{P})$ and $\alpha^k(Y|_{P})$. Since $P$ is assumed to contain the supports of the Betti diagrams of $X$ and $Y$, the 
 natural transformations $\alpha^k(X|_{P})\to X$
 and  $\alpha^k(Y|_{P})\to Y$ are isomorphisms. This, together with conditions (i) and (ii) above  imply that $X(v\leq w)$ and $Y(v\leq w)$ are isomorphisms  for every $v\leq w$ in $[0,\infty)^\ell$ for which there is $p$ in $P$ such that
$p\leq v\leq w\leq p\ast (2\eta)$. 

Let us choose natural transformations  $f\colon X\to Z$ and $g\colon Y\to Z$ such that $N(f)\leq \eta$ and   $N(g)\leq \eta$.
Their existence is guaranteed by the assumption $d_N(X,Y)=0$. The norms   $N(\text{ker}(f))$,  
 $N(\text{coker}(f))$, $N(\text{ker}(g))$, and   
 $N(\text{coker}(g))$ are bounded above by $\eta$.
 It follows that,  for any $v$ in $[0,\infty)^\ell$ for which $v\ast\eta<\infty$ and $X(v\leq v\ast\eta)$ is an isomorphism, the function $f_v\colon X(v)\to Z(v)$ is a monomorphism. Indeed, since  $N(\text{ker}(f))\leq \eta$, any element in the kernel of $f_v$ would be  mapped via the isomorphism  $X(v\leq v\ast\eta)$ to  $0$, and hence this element has to be $0$.  In particular,  for every $p$ in $P$, the functions $f_p\colon X(p)\to Z(p)$ and
 $f_{p\ast\eta}\colon X(p\ast\eta)\to Z(p\ast\eta)$ are monomorphisms.
 Same arguments give that $g_p\colon Y(p)\to Z(p)$ and $g_{p\ast\eta}\colon Y(p\ast\eta)\to Z(p\ast\eta)$ are also monomorphisms.

 Let us organize these maps into the following commutative diagram, ignoring for now the dotted arrows, where the indicated functions are monomorphisms and isomorphisms:
 \[\begin{tikzcd}[row sep=28pt]
     X(p)\ar[hook]{r}{f_{p}} 
     \ar{d}{\simeq}[swap]{X(p\leq p\ast\eta)} & Z(p)\ar{d}[description]{Z(p\leq p\ast\eta)}\ar[dotted]{dl}[swap]{s_1}
     \ar[dotted]{dr}{s_2} & Y(p)\ar[hook']{l}[swap]{g_{p}}
     \ar{d}{Y(p\leq p\ast\eta)}[swap]{\simeq}
     \\
     X(p\ast\eta)\ar[hook]{r}{f_{p\ast\eta}} & Z(p\ast\eta)
     & Y(p\ast\eta).\ar[hook']{l}[swap]{g_{p\ast\eta}}
 \end{tikzcd}\]
 Since 
$N(\text{coker}(f))$ and $N(\text{coker}(g))$ are bounded by $\eta$, the maps $\text{coker}(f)(p\leq p\ast\eta)$ and 
 $\text{coker}(g)(p\leq p\ast\eta)$ are trivial. Consequently, the image of the map $Z(p\leq p\ast\eta)$ lies in the images of both
$f_{ p\ast\eta}$ and $g_{ p\ast\eta}$, and hence the existence of 
the maps represented by the dotted arrows and  making the entire diagram commutative. The compositions $s_1 g_p$ and $s_2 f_p$ are therefore monomorphisms, and since the vector spaces are finite dimensional, they have to be isomorphisms.  The maps
given by the compositions
$\{X(p\leq p\ast\eta)^{-1}s_1g_p\colon Y(p)\to X(p)\}_{p\in P}$
form  a natural isomorphism between the restrictions $Y|_{P}$ and 
 $X|_{P}$. Its left Kan extension along $\alpha$ is the desired isomorphism between $X$ and $Y$. 
\end{proof}

    \begin{point}
    In addition to the distance $d_N$, the contour $\ast$
     can be used to define a so called \textbf{interleaving distance}, which is a more traditional way of associating a distance to a contour (see for example~\cite{MR4172343, MR3413628}). Let us recall its definition. For 
 $\varepsilon$ in $[0,\infty)$, define the \textbf{$\varepsilon$-translation} of a tame functor 
  $X\colon [0,\infty)^\ell\to\text{vec}_K$, to be the tame functor 
 $X^\varepsilon\colon [0,\infty)^\ell\to\text{vec}_K$ given by the formula:
 \[X^\varepsilon(v\leq w)=\begin{cases}
     X(v\ast \varepsilon\leq w\ast\varepsilon) & \text{ if } w\ast\varepsilon <\infty \\
     X(v\ast \varepsilon)\to 0 & \text{ if } v\ast\varepsilon < w\ast\varepsilon =\infty \\
     0\to 0 &\text{ if }  v\ast\varepsilon=\infty.
 \end{cases}\]
 The maps $X(v\leq v\ast \varepsilon)$, if $v\ast \varepsilon<\infty$, and $X(v)\to 0$, if $v\ast \varepsilon=\infty$,
 indexed by $v$ in $[0,\infty)^\ell$, form a natural transformation denoted by $\mu_{X}\colon X\to X^\varepsilon$.
 
 If $f\colon X\to Y$ is a natural transformation, then the morphisms $f(v\ast \varepsilon)\colon X(v\ast \varepsilon)\to Y(v\ast \varepsilon)$, if $v\ast \varepsilon<\infty$, and
 $0\colon 0\to 0$, if $v\ast \varepsilon=\infty$, indexed by $v$ in $[0,\infty)^\ell$,  form a natural transformation denoted by
 $f^\varepsilon\colon X^\varepsilon\to Y^\varepsilon$. 
 
   Tame functors $X,Y\colon [0,\infty)^\ell\to\text{vec}_K$ are called \textbf{$\varepsilon$-interleaved } (with respect to $\ast$) if there are natural transformations
 $f\colon X\to Y^\varepsilon$ and $g\colon Y\to X^\varepsilon$  making the following diagrams commutative:
 \[
 \begin{tikzcd}[row sep=1pt]
      &Y^\varepsilon\ar[bend left=8pt]{dr}{g^\varepsilon} \\
X\ar[bend left=8pt]{ru}{f}\ar[bend right=8pt]{dr}[swap]{\mu_{X}} & & (X^\varepsilon)^\varepsilon\\
& X^\varepsilon\ar[bend right=8pt]{ur}[swap]{\mu_{X^{\varepsilon}}}
 \end{tikzcd}\ \ \ \ \ \ \ 
 \begin{tikzcd}[row sep=1pt]
      &X^\varepsilon\ar[bend left=8pt]{dr}{f^\varepsilon} \\
Y\ar[bend left=8pt]{ru}{g}\ar[bend right=8pt]{dr}[swap]{\mu_{Y}} & & (Y^\varepsilon)^\varepsilon\\
& Y^\varepsilon\ar[bend right=8pt]{ur}[swap]{\mu_{Y^{\varepsilon}}}
 \end{tikzcd}
 \]
The interleaving distance is then defined as:
\[d_{\text{int}}\coloneqq\inf\{\varepsilon \geq 0\mid X,Y \text{ are $\varepsilon$-interleaved}\}.\]
\end{point}
\begin{prop} For tame functors $X,Y\colon [0,\infty)^\ell\longrightarrow \vect$ we have that: \[d_{\text{\rm int}}(X,Y)\leq d_N (X,Y) \leq 4\cdot d_{\text{\rm int}}(X,Y).\] 
\end{prop}
\begin{proof}
    Let  $\omega_\ast$ be the simple valuation associated with the contour $\ast$ as defined in~\ref{sadgsfgj}.
    Consider a natural transformation $f\colon X\to Y$. Then, for every $t$ in $c_f$ (see~\ref{sadgsfgj}), the lifts at $t$ form a natural transformation and hence, together with $f$, give
    a $t$-interleaving between $X$ and $Y$. Thus $d_{\text{int}}(X,Y)\leq \omega_\ast(f)\leq N(f)$ (see~\ref{adsgdsjgk}.(4)).  
    This implies the left inequality of the statement of the proposition.

    To show the second inequality, let
    $f\colon X\to Y^\varepsilon$ and $g\colon Y\to X^\varepsilon$,   be natural transformations forming an $\epsilon$-interleaving between $X$ and $Y$. Consider the following  pull-back square:
    \[
\begin{tikzcd}
    P\ar{r}{\alpha}\ar{d}[swap]{\beta}& X\ar{d}{\begin{bsmallmatrix}\mu_X\\f\end{bsmallmatrix}}\\
    Y\ar{r}{\begin{bsmallmatrix}g \\ \mu_Y\end{bsmallmatrix}} & X^{\varepsilon} \oplus Y^\varepsilon.
\end{tikzcd}
\] 
The fact that $f$ and $g$ form $\varepsilon$-interleaving, implies that 
$\varepsilon$ belongs to both $c_{\alpha}$ and  $c_{\beta}$, and hence
 $\omega_\ast(\alpha)\leq \varepsilon$ and $\omega_\ast(\beta)\leq \varepsilon$.
Consequently,  $d_{\omega_\ast}(X,Y)\leq \omega_{\ast}(\alpha)+\omega_{\ast}(\beta)\leq 2 \varepsilon$. As this happens for every such $\varepsilon$, we get
$d_{\omega_\ast}(X,Y)\leq 2 d_{\text{int}}(X,Y)$. This, combined with~\ref{adsgdsjgk}.(5), gives the desired right inequality. 
\end{proof}

\section{Continuous and  Lipschitz contours}
According to~\ref{asfhgfgdjgh}, 
sets of the form $B_{N}(v,t)\coloneqq\{w\in [0,\infty)^\ell \mid  \delta(v,w)+\delta(w,v)<t\}$,
for $v$ in $[0,\infty)^\ell$ and $t$ in $(0,\infty)$, 
provide a base for the topology on   $[0,\infty)^\ell$ induced by the distance $d_N$. We can use the step function $\delta$ to construct other examples of open  sets in this  topology. 
\begin{prop}\label{asdfgsdh}
    Let $v$ be in $[0,\infty)^\ell$ and $t$ in $(0,\infty)$. The following sets are open in the topology on   $[0,\infty)^\ell$ induced by the distance $d_N$:
        \[
        \{w\in [0,\infty)^\ell\mid \delta(v,w)<t\}\ \ \ \ \ \ 
        \{w\in [0,\infty)^\ell\mid \delta(w,v)<t\}\] 
        \[\{w\in [0,\infty)^\ell\mid \delta(v,w)>t\}\ \ \ \ \ \ 
        \{w\in [0,\infty)^\ell\mid \delta(w,v)>t\}.
        \]
\end{prop}
\begin{proof}
We claim that, if $w_0,w_1$ in $[0,\infty)^\ell$ are such that  $d_N(w_0,w_1)$ is finite, then:
\begin{enumerate}
\item $\delta(v,w_0)$ is finite if and only if 
$\delta(v,w_1)$ is finite, in which case:
\[d_N(w_0,w_1)\geq |\delta(v,w_0)-\delta(v,w_1)|;\]
\item $\delta(w_0, v)$ is finite if and only if 
$\delta(w_1,v)$ is finite, in which case:
\[d_N(w_0,w_1)\geq |\delta(w_0, v)-\delta(w_1, v)|.\]
\end{enumerate}

Since $d_N(w_0,w_1)$ is finite, then so is $\delta(w_0,w_1)$.
By ~\ref{asghfsgjfghdkj}.(3),  $\delta(v,w_1)\leq \delta(v,w_0)+ \delta(w_0,w_1)\leq 
 \delta(v,w_0)+ d_N(w_0,w_1)$, and  thus $\delta(v,w_1)$ is finite if 
$\delta(v,w_0)$ is finite. By symmetry, if $\delta(v,w_1)$ is finite, then also 
  $\delta(v,w_0)\leq \delta(v,w_1)+ \delta(w_1,w_0)\leq 
 \delta(v,w_1)+ d_N(w_0,w_1)$. Both of these inequalities imply the first claim. The second claim can be proven in the same way. 

 These two statements state that the step function $\delta$ is Lipschitz in both variables. We can now argue that, for every $w$ such that $\delta(v,w)<t$, according to claim (1), any element $u$ in $B_N(w, t-\delta(v,w))$ also satisfies the inequality $\delta(v,u)<t$.
 Consequently, the set $\{w\in [0,\infty)^\ell\mid \delta(v,w)<t\}$ is  open  in the topology induced by $d_N$.  Same is true for
 the set $\{w\in [0,\infty)^\ell\mid \delta(w,v)<t\}$. 

 Let $w$ be such that $\delta(v,w)>t$. Choose $\varepsilon$ in $(0,\infty)$ for which  $\delta(v,w)>t+\varepsilon$. Then for any element $u$ in $B_N(w, \varepsilon)$, according to claim (1), $\delta(v,u)>t$.  Consequently, the set $\{w\in [0,\infty)^\ell\mid \delta(v,w)>t\}$ is open  in the topology induced by $d_N$.
 Similar argument gives the openness of $\{w\in [0,\infty)^\ell\mid \delta(w,v)>t\}$.
\end{proof}

Sets in~\ref{asdfgsdh} can be used to form a 
subbase 
(every open set is a union of finite intersections) of the topology induced by $d_N$:
\begin{prop}\label{sasfddh}
Sets of the form $\{w\in [0,\infty)^\ell\mid \delta(v,w)<t\}$ or 
    $\{w\in [0,\infty)^\ell\mid \delta(w,v)<t\}$, for $v$ in   $[0,\infty)^\ell$ and $t$ in $(0,\infty)$,  form a subbase of the 
    topology on   $[0,\infty)^\ell$ induced by the distance $d_N$.
\end{prop}
\begin{proof}
This is a consequence of~\ref{asdfgsdh} and the equality:
\[B_N(v,t)=\{w\in [0,\infty)^\ell \mid \delta(v,w)+\delta(w,v)<t\}=\]
\[=\bigcup_{\substack{a+b<t\\ a> 0, b> 0}}
\{w\in [0,\infty)^\ell \mid  \delta(v,w)<a\}\cap \{w\in [0,\infty)^\ell \mid  \delta(w,v)<b\}.\qedhere\]
\end{proof}

\begin{point}
In addition to the distance $d_N$ induced by the norm $N$, the set    $[0,\infty)^\ell$  is  equipped with more standard distances, for example:
\begin{enumerate}
    \item
    $d_1(v,w)=|v_1-w_1|+\cdots + |v_\ell-w_\ell|$;
    \item 
        $d_2(v,w)=\sqrt{(v_1-w_1)^2+\cdots + (v_\ell-w_\ell)^2}$;
    \item 
        $d_\infty(v,w)=        \text{max}\{|v_i-w_i| \mid 1\leq i\leq \ell\}$.
\end{enumerate}
Recall that the following identity functions are $1$-Lipschitz:
\[
\begin{tikzcd}
    ([0,\infty)^\ell, d_1)\ar{r}{\text{id}} &
    ([0,\infty)^\ell, d_2)\ar{r}{\text{id}} &
    ([0,\infty)^\ell, d_\infty).
\end{tikzcd}\]
In particular, all these standard distances induce the same topology on
$[0,\infty)^\ell$ called standard. In this topology, a subset $T\subset [0,\infty)^\ell$ is compact (any open cover has a finite sub-cover) if and only if it is closed and bounded. Moreover, since these metrics are complete, any such compact set is totally bounded~\cite{MR385023}.
\end{point}
\begin{point}\label{asvdgg}
    The contour $\ast$ is called \textbf{continuous} if the identity function 
    $\mathsf{id}\colon ([0,\infty)^\ell, d_\infty)\to ([0,\infty)^\ell, d_N)$ is continuous. According to~\ref{sasfddh}, this happens if and only if,
    for every $v$ in $[0,\infty)^\ell$ and $t$ in $(0,\infty)$, sets of the form 
    $\{w\in [0,\infty)^\ell\mid \delta(v,w)<t\}$ and  
    $\{w\in [0,\infty)^\ell\mid \delta(w,v)<t\}$ are open in the standard topology. 
    
     Let $C$ be a real number. The contour  $\ast$ is called \textbf{$C$-Lipschitz} if 
    $C d_\infty(v,w)\geq  d_N(v,w)$ for all $v$ and $w$ in $[0,\infty)^\ell$. It is called  \textbf{Lipschitz} if it is 
    $C$-Lipschitz for some $C$.
    Lipschitz contours are  continuous.
    For example, if $\ast$ is the standard contour, then:
    \[
\begin{aligned}
2d_{\infty}(v,w)
&=2\max\{|v_i-w_i| \mid 1\leq i\leq \ell\} \geq\\
&\geq
\max\{0, w_1-v_1,\ldots, w_\ell-v_\ell\} +\max\{0, v_1-w_1,\ldots, v_\ell-w_\ell\} =\\
&=\delta(v,w)+\delta(w,v)
=d_N(v,w).
\end{aligned}
\]
    
     Thus the standard contour is $2$-Lipschitz and therefore also continuous.
\end{point}
\begin{point}\label{asfghdfgjhdghjnghm}
Assume the contour $\ast$ is continuous. 
For $t$ in $(0,\infty)$ and $v$ in $[0,\infty)^\ell$, since the set $B_{N}(v,t)$ is open in the standard topology and contains $v$, it also contains some $v'$ for which $v'\gg v$ (i.e., $(v')_i>v_i$ for every $i$, see~\ref{dfgsfdhfg}). For this element  $\delta(v,v')\leq d_N(v,v')<t$ and hence $v\ast t\geq v'\gg v$. We call this property of the contour  \textbf{being strictly increasing in all coordinates}:  $v\ast t\gg v$ for every $v$ in  $[0,\infty)^\ell$ and $t$ in $(0,\infty)$. Being strictly increasing in all coordinates is therefore a necessary condition for a contour to be continuous  (compare with~\cite{bauer2026metricallycompletekrullschmidtspace}).   For example, assume $\ast$ is the standard contour and consider the contour $\ast_f$ 
(see~\ref{lbacmgocbd}), where $f$ is defined  by the formula:   
\[f(t)\coloneqq\begin{cases}
0 &\text{ if } t\leq 1\\
t-1 & \text{ if } t\geq 1.
\end{cases}\]
Then $v \ast_f 1 = v \ast f(1) = v \ast 0=v$. Thus $\ast_f$ fails to be strictly increasing and hence is not continuous (see~\ref{asvdgg}).

Here is a sufficient condition for a  contour to be continuous:
\end{point}
\begin{prop}\label{asdgdfhfg}
Assume the contour $\ast$ is strictly increasing in all coordinates and satisfies the following property:
    for all $v$ in $[0,\infty)^\ell$ and  $t$ in $(0, \infty)$, the set $\{w \in [0,\infty)^\ell \mid w \ast t \gg v\}$ is open in the standard topology. Then $\ast$ is continuous.
\end{prop}
\begin{proof}
Choose $v$ in 
$[0,\infty)^\ell$ and  $t$ in $(0,\infty)$.
We first claim:
\[
\{w\mid \delta(v,w)<t\}=\bigcup_{0 < s<t}\{w\mid v\ast s\gg w\},\ \ \ \ \ \ \ \ \ \ \ \{w\mid \delta(w,v)<t\}=\bigcup_{0 < s<t}\{w\mid w\ast s\gg v\}.
\]
If $\delta(v,w)<t$, pick $s<t$ with $v\ast s\ge w$. For $\varepsilon>0$ with $s+\varepsilon<t$,
\[
v\ast(s+\varepsilon)\ge (v\ast s)\ast\varepsilon \ge w\ast\varepsilon \gg w,
\]
so $v\ast(s+\varepsilon)\gg w$. Conversely, $v\ast s\gg w$ implies $\delta(v,w)\le s<t$. Hence the first equality of the  claim.
The second equality can be proven by an analogous argument.

For fixed $s$, the set $\{w\mid v\ast s\gg w\}=\bigcap_i\{w\mid w_i<(v\ast s)_i\}$ is open in the standard topology;  the set $\{w\mid w\ast s\gg v\}$ is open in the standard topology by assumption. Thus according to the claim, the sets $\{w\mid \delta(v,w)<t\}$  and $ \{w\mid \delta(w,v)<t\}$ are also  open in the standard topology. As this happens for every $v$ and $t$,   the contour is continuous.
\end{proof}

Next we discuss under what circumstances the constructions introduced in sections~\ref{sdgdghj}-\ref{lbacmgocbd} preserve   continuity and Lipschitzness of contours.
We can then apply them to 
for example the standard contour to produce a plethora of other Lipschitz and continuous contours.

\begin{prop}\label{asfhgfgjghjkg}\hspace{2mm}
\begin{enumerate}
    \item Consider an 
    order isomorphism $\varphi\colon [0,\infty)^\ell\to [0,\infty)^\ell$ 
    (see~\ref{sdgdghj}). 
    \begin{itemize}
        \item[(a)] If the contour $\ast$ is continuous and 
        $\varphi$ continuous in the standard topology, then the contour $\ast^\varphi$
        (see~\ref{sdgdghj}) is also continuous.
        \item[(b)] If the contour $\ast$ is $C$-Lipschitz, and $\varphi$  is $D$-Lipschitz with respect to $d_\infty$, then 
        the contour $\ast^\varphi$
        is $CD$-Lipschitz.
    \end{itemize}
    \item Consider a finite subset $\Gamma\subset [0,\infty)^\ell$. If 
    the contour $\ast$ is continuous (resp.\ $C$-Lipschitz), then so is  $\ast_{/\Gamma}$ (see~\ref{asdfdfhdghj}).
    \item Consider an element $x$ in $[0,\infty)^\ell_\infty$. If 
     the contour $\ast$ is continuous (resp.\ $C$-Lipschitz), then so is $\ast_{\vee x}$ (see~\ref{dsdgdfh}). 
    \item Consider a contour type function $f\colon[0,\infty)\to[0,\infty]$ (see~\ref{lbacmgocbd}).
    \begin{itemize}
        \item[(a)] If the contour $\ast$ is continuous and 
        $f(t) > 0$ for all $t > 0$, then 
        the contour $\ast_f$ (see~\ref{lbacmgocbd}) is also continuous.
        \item[(b)] If the contour $\ast$ is $C$-Lipschitz and 
        there is $D$ such that $Df(t)\geq t$ for all $t>0$, then the contour $\ast_f$ is $CD$-Lipschitz.
    \end{itemize}
    \end{enumerate}
\end{prop}
\begin{proof}
(1):\quad
To show (a), note that according to~\ref{sdgdghj}, for  $v$ in $[0,\infty)^\ell$ and $t$ in $(0,\infty)$, we have:
    \[\{w\in [0,\infty)^\ell \mid  \delta^\varphi(v,w)+\delta^\varphi(w,v)<t\}=\{w\in [0,\infty)^\ell \mid  \delta(\varphi(v),\varphi(w))+\delta(\varphi(w),\varphi(v))<t\}=\]
    \[=\varphi^{-1}(\{w\in [0,\infty)^\ell \mid  
    \delta(\varphi(v),w) + \delta(w,\varphi(v))<t\}=
    \varphi^{-1}(B_{N}(\varphi(v), t)),\]
    which is open in the standard topology, since $\varphi$ is assumed to be continuous. 

    For (b), note that:
    \begin{align*}
    d_{N^{\varphi}}(v,w)=\delta^\varphi(v,w)+\delta^\varphi(w,v)=
    \delta(\varphi(v),\varphi(w)) + \delta(\varphi(w),\varphi(v))=\\= d_{N}(\varphi(v),\varphi(w))\leq C d_\infty(\varphi(v),\varphi(w))\leq 
    C D d_{\infty}(v,w).\end{align*}

\smallskip

    \noindent
(2):\quad Recall  $\delta_{/\Gamma}(v,w)=\text{min}\{\delta(v,w), \delta(v,x)\mid x\in \Gamma\} \leq \delta(v,w)$. 
This means that  
the function
$\text{\rm id}\colon ([0,\infty)^\ell, d_{N_{\ast}})\to ([0,\infty)^\ell, d_{N_{\ast_{/\Gamma}}})$ is 1-Lipschitz,
and the statement follows. 
\smallskip

\noindent
(3):\quad 
By \ref{asghfsgjfghdkj}.(2), as $v \leq v\vee x$, we have $\delta_{\vee x}(v,w)=\delta(v\vee x, w) \leq \delta(v, w)$, and the rest of the  argument  is as in (2).
\smallskip

\noindent
(4):\quad 
To show (a), note that the assumptions on $f$  give its strict monotonicity:
if $t > s$, then $f(t) = f(s + (t-s) \geq f(s) + f(t-s) > f(s)$.
Thus according to ~\cite[Proposition 1.(7)]{MR3072795},   its  generalized inverse $f^{-1}\colon [0,\infty]\to [0,\infty]$ ($f^{-1}(a) := \inf\{t\in[0,\infty)\mid f(t)\geq a\}$) is continuous on the range of $f$. The right continuity of $f$ implies further that  $f^{-1}$
is continuous on the entire $[0,\infty]$
 In particular, for every $t$ in $(0,\infty)$, the set  $A_t \coloneqq \{a \in [0, \infty]\mid f^{-1}(a) < t \}$ 
 is an open  interval of the form $[0, e_t)$. As   $\delta_{\ast_f}(v,w) = f^{-1}\bigl(\delta(v,w)\bigr)$  (see~\ref{lbacmgocbd}), we get:
\[\{w\in [0,\infty)^\ell\mid \delta_{\ast_f}(w,v)<t\} =\{w\in [0,\infty)^\ell\mid f^{-1}(\delta(w,v))<t\} = \{w\in [0,\infty)^\ell\mid \delta(w,v) < e_t \}, \]
\[\{w\in [0,\infty)^\ell\mid \delta_{\ast_f}(v,w)<t\} =\{w\in [0,\infty)^\ell\mid f^{-1}(\delta(v,w))<t\} = \{w\in [0,\infty)^\ell\mid \delta(v,w) < e_t \}. \]
 By the continuity of the contour $\ast$, these subsets of $[0,\infty)^\ell$ are  open in the standard topology, and consequently  the contour
$\ast_f$ is also continuous (see~\ref{sasfddh}). 

For (b),  note that the inequality $f(t)\geq t/D$ implies 
$t\geq f^{-1}(t/D)$. As this happens for every $t>0$, we  also have
$Dt\geq f^{-1}(t)$. Thus $D\delta(v,w)\geq f^{-1}(\delta(v,w))=\delta_{\ast f}(v,w)$, and consequently:
\[CD d_{\infty}(v,w) \geq Dd_{N_\ast}(v,w)= D\delta(v,w) +  D\delta(w,v)\geq  \delta_{\ast f}(v,w) + \delta_{\ast f}(w,v)=d_{N_{\ast f}}(v,w).\qedhere\]
\end{proof}

\section{Compact sets of  tame representations of $[0,\infty)^\ell$}
Recall that we fixed a simple norm $N$, the associated contour $\ast$, and its step function $\delta$ (see the beginning of Section~\ref{asdfdfh}). 

Throughout this section the term compact space refers to a topological space whose every cover by open sets admits a finite sub-cover. Recall that in the case the topology is given by a metric, a space is compact if, and only if, it is sequentially compact, i.e. if any sequence contains a convergent subsequence (\cite{MR385023}). This may fail however 
if  the topology is induced by a pseudo-metric. In particular, 
a compact subset  of a pseudo-metric space  may fail to be closed. 
However any 
 pseudo-metric $d$ on a space $\mathcal{D}$, leads to a metric $\Tilde{d}$ on 
 the quotient $\mathcal{D}/\!\sim$, where $x\sim y$ if and only if $d(x,y)=0$ and  $\Tilde{d}([x],[y])=d(x,y).$
One can then show that a subset $X$ of $\mathcal{D}$ is compact if, and only if, its image in $\mathcal{D}/\!\sim$ is compact, i.e., if this image is sequentially compact, which happens if, and only if, the subset $\{y\in \mathcal{D}\mid d(x,y)=0 \text{ for some $x$ in $X$}\}$ of $\mathcal{D}$ is compact.

\begin{point}\label{sadgdghjf}
For a  natural number $m>0$, elements in 
$([0,\infty)^\ell)^m$ are going to be denoted as sequences $v=(v_1,\ldots, v_m)$.

We are going to consider two metrics on $([0,\infty)^\ell)^m$  given by the Manhattan extensions of $d_\infty$ and $d_N$ on $[0,\infty)^\ell$:
\[d_N(v,w)=\sum_{i=1}^m d_N(v_i,w_i)=\sum_{i=1}^m\left(\delta(v_i,w_i)+\delta(w_i,v_i)\right),\]
\[d_\infty(v,w)=\sum_{i=1}^m d_\infty(v_i,w_i)=
\sum_{i=1}^m \text{max}\{|(v_{i})_j-(w_{i})_j| \mid 1\leq j\leq l\}. \]
We refer to the topology on $([0,\infty)^\ell)^m$ induced by the metric $d_\infty$ as the standard topology. 
If the contour $\ast$ is continuous, then  the identity function $\mathsf{id}\colon
(([0,\infty)^\ell)^m,d_\infty)\to (([0,\infty)^\ell)^m,d_N)$ is continuous (see~\ref{asvdgg}).

We also consider $([0,\infty)^\ell)^m$  as a poset with the product relation: $(v_1,\ldots, v_m)\geq (w_1,\ldots, w_m)$ in $([0,\infty)^\ell)^m$ if and only if $v_i\geq w_i$ in
$[0,\infty)^\ell$ for every $1\leq i\leq m$. With this poset relation $([0,\infty)^\ell)^m$ is a distributive lattice with the meet and join given by $(v_1,\ldots, v_m)\wedge (w_1,\ldots, w_m)=(v_1\wedge w_1,\ldots, v_m\wedge  w_m)$ and $(v_1,\ldots, v_m)\vee (w_1,\ldots, w_m)=(v_1\vee w_1,\ldots, v_m\vee  w_m)$. 

\end{point}

\begin{point}\label{sdgdfhgh}
Choose an $n \times m$ matrix $P$ with coefficients in the field $K$ and define:
\[\mathcal{S}_P:=\{(v,w)\in ([0,\infty)^\ell)^m\times ([0,\infty)^\ell)^n \mid 
\text{ if $P_{i,j}\not=0$ then $v_j\geq w_i$ in $[0,\infty)^\ell$}\}.\]
For example, for the zero matrix,  $\mathcal{S}_0=([0,\infty)^\ell)^m\times ([0,\infty)^\ell)^n$. In general,
the set $\mathcal{S}_P$ is described by a system of linear inequalities $(v_{j})_k\geq (w_{i})_{k}$ indexed by nontrivial entries $P_{i,j}$ of the matrix $P$ and $1\leq k\leq l$. Thus,
$\mathcal{S}_P \subset([0,\infty)^\ell)^m\times ([0,\infty)^\ell)^n$   is a cone.  In particular it is a convex subset and it is closed in the standard topology.
Its
interior in the standard topology is described by the system of analogous but strict linear inequalities $(v_{j})_k> (w_{i})_{k}$ and is never empty. For example, for an element $w$ in $([0,\infty)^\ell)^n$, if  $v$ in $([0,\infty)^\ell)^m$ is such that  $(v_{i})_k> \text{max}\{(w_{j})_s \mid 1\leq j\leq n\text{ and } 1\leq s\leq l\}$, for every $1\leq i\leq m$ and $1\leq k\leq l$, then the pair $(v,w)$ belongs to the interior of $\mathcal{S}_P$.

Since taking meets and joins preserve the poset relation, 
$\mathcal{S}_P$ enjoys  the following closure property: if $(v,w)$ and $(v',w')$ belong to $\mathcal{S}_P$, then so do  $(v\wedge v',w\wedge w')$ and $(v\vee v', w\vee w')$.

By restricting the metrics $d_N$ and $d_\infty$
on $([0,\infty)^\ell)^m\times ([0,\infty)^\ell)^n=([0,\infty)^\ell)^{m+n}$ (see~\ref{sadgdghjf}), we  endow $\mathcal{S}_P$ with two metrics  denoted also by $d_N$ and $d_\infty$.

    Recall  that the norm $N$ on $\text{\rm tame}([0,\infty)^\ell, \text{\rm vec}_K)$ leads also to a valuation 
    on the category $\text{Fun}([1], \text{\rm tame}([0,\infty)^\ell, \text{\rm vec}_K))$ (see~\ref{iphfinpxla}), and hence 
    a pseudo-metric $d_N$ on its objects. 

For $(v,w)$ in $\mathcal{S}_P$, define an object $\text{\rm res}_P(v,w)$ in  $\text{Fun}([1], \text{\rm tame}([0,\infty)^\ell, \text{\rm vec}_K))$ as follows:
\[\text{\rm res}_P(v,w):= 
    \left( \begin{tikzcd} \displaystyle
    \bigoplus_{j=1}^m K(v_j,-)\ar{r}{P} &
    \displaystyle\bigoplus_{i=1}^n K(w_i,-)
    \end{tikzcd}\right).
    \]

\end{point}
\begin{prop}\label{sgdghjfgj}
Let $P$ be a $n \times m$ matrix with coefficients in  the field $K$. For all pair of elements $(v,w)$ and $(v',w')$ in $\mathcal{S}_P$:
\[d_N((v,w),(v',w'))\geq d_N(\text{\rm res}_P(v,w),\text{\rm res}_P(v',w')).\]
\end{prop}
This proposition states that the function $\text{\rm res}_P\colon
\mathcal{S}_P\to \text{\rm Fun}([1], \text{\rm tame}([0,\infty)^\ell, \text{\rm vec}_K)$ is   $1$-Lipschitz with respect to the distances $d_N$.
\begin{proof}
Consider the following commutative diagram:
    \[
\begin{tikzcd}[row sep=2em, column sep=4em]
\bigoplus_{i=1}^m K(v_i, -) \arrow[d, "P"'] 
  & \bigoplus_{i=1}^m K(v_i \vee v_i', -) \arrow[l, hook', "\phi_0"'] \arrow[r, hook, "\psi_0"] \arrow[d, "P"'] 
  & \bigoplus_{i=1}^m K(v_i', -) \arrow[d, "P"] \\
\bigoplus_{i=1}^n K(w_i, -) 
  & \bigoplus_{i=1}^n K(w_i \vee w_i', -) \arrow[l, hook', "\phi_1"'] \arrow[r, hook, "\psi_1"] 
  & \bigoplus_{i=1}^n K(w_i', -)
\end{tikzcd}
\]
where the horizontal rows are direct sums of $K(v_i,-)\supset K(v_i\vee v'_i,-)\subset K(v'_i,-)$ and $K(w_i,-)\supset K(w_i\vee w'_i,-)\subset K(w'_i,-)$. Since the norm  $N$ is simple:
\[N(\phi_0)\leq \text{max}\{\delta(v'_i, v_i) \mid 1\leq i\leq m\},\ \ \ \ \ N(\psi_0)\leq \text{max}\{\delta(v_i, v'_i) \mid 1\leq i\leq m\},\]
\[N(\phi_1)\leq \text{max}\{\delta(w'_i, w_i) \mid 1\leq i\leq n\},\ \ \ \ \ N(\psi_1)\leq \text{max}\{\delta(w_i, w'_i) \mid 1\leq i\leq n\}.\]
Thus:
\[N(\phi_0,\phi_1)= N(\phi_0)+N(\phi_1)\leq 
\text{max}\{\delta(v'_i, v_i) \mid 1\leq i\leq m\} + \text{max}\{\delta(w'_i, w_i) \mid 1\leq i\leq n\},\]
\[N(\psi_0,\psi_1)= N(\psi_0)+N(\psi_1)\leq 
\text{max}\{\delta(v_i, v'_i) \mid 1\leq i\leq m\} + \text{max}\{\delta(w_i, w'_i) \mid 1\leq i\leq n\}.\]
Which gives:
\[
d_N((v,w),(v',w'))=d_N(v,v')+d_N(w,w')=
\sum_{i=1}^{m}(\delta(v'_i,v_i)+\delta(v_i,v'_i))+\sum_{i=1}^{n}(\delta(w'_i,w_i)+\delta(w_i,w'_i))\geq
\]
\[
\geq N(\phi_0,\phi_1) + N(\psi_0,\psi_1)\geq
d_N(\text{\rm res}(v,w), \text{\rm res}(v',w')).\qedhere
\]
\end{proof}

\begin{coro}\label{adgjkk}
    For every $n \times m$ matrix $P$ with coefficients in  the field $K$, the following composition is a $1$-Lipschitz function with respect to the distances $d_N$:   
    \[\begin{tikzcd}
        \mathcal{S}_P\ar{r}{\text{\rm res}_P} & \text{\rm Fun}([1], \text{\rm tame}([0,\infty)^\ell, \text{\rm vec}_K))\ar{r}{\text{\rm coker}}&
\text{\rm tame}([0,\infty)^\ell, \text{\rm vec}_K).
    \end{tikzcd}\]
\end{coro}
\begin{proof}
    This is a consequence of Proposition~\ref{sgdghjfgj} and 
    Proposition~\ref{tyjrytujt}.(2).
\end{proof}

We are now ready to state our final compactness theorem.
For a subset $D\subset [0,\infty)^\ell$, and two positive integers $n$ and $m$, define:
 \[
\mathcal{T}(D,n,m)\coloneqq \left\{\begin{array}{c}\text{isomorphism classes of tame functors}\\ X\colon [0,\infty)^\ell\to \vect
\end{array}\;\Biggm|\; 
\begin{array}{c}
\text{\rm supp}(\beta_0X)\subset D\supset \text{\rm supp}(\beta_1X),\\
\sum\beta_0 X\leq n\text{ and } \sum\beta_1 X\leq m
\end{array}
\right\}.
 \]

\begin{thm}\label{sadfhfggjh} 
Let $\ast\colon [0,\infty)^\ell_\infty\times [0,\infty)\to [0,\infty)^\ell_\infty$ be a contour and $N$   the  simple norm on   $\text{\rm tame}([0,\infty)^\ell, \text{\rm vec}_K)$ associated with $\ast$ (see G\"afvert's Theorem~\ref{sadgfghjg}).
Assume $K$ is a finite field,  $D\subset [0,\infty)^\ell$ is a compact subset in the standard topology, and $m,n$ positive integers. 
\begin{enumerate}
    \item If the contour $\ast$ is continuous (see~\ref{asvdgg}), then  the set $\mathcal{T}(D,n,m)$, with the topology induced by the pseudo-metric $d_N$, is compact. 
    \item If the contour $\ast$ is continuous  and unital (see~\ref{assfadfhfgs}), then the set $\mathcal{T}(D,n,m)$, with the topology induced by the metric $d_N$ (see~\ref{asdgadsfth}), is compact, complete,  and totally bounded (see~\cite{MR385023}).
\end{enumerate}
\end{thm}
\begin{proof}
Let $\mathcal{S}_{P, D}\coloneqq\{(v,w)\in  \mathcal{S}_{P}\mid v_i,w_j\in D \text{ for } 1\leq i\leq m \text{ and } 1\leq j\leq n\}$ (see~\ref{sdgdfhgh}).  With the topology induced by the metric $d_\infty$, this set $\mathcal{S}_{P, D}$ as a subspace of $([0,\infty)^\ell)^m\times ([0,\infty)^\ell)^n$,   is closed, bounded, and therefore compact, as well as complete and totally bounded.
Moreover, since $K$ is assumed to be a finite field, there are only finitely many $n\times m$ matrices and consequently, 
the disjoint union $\coprod_{n\times m \text{ matrices }P} \mathcal{S}_{P, D}$ is also compact, complete and totally bounded.

Consider the following composition denoted by $\Psi$:
    \[
    \begin{tikzcd}[column sep=15pt, row sep=30pt]
      \displaystyle{\coprod_{n\times m \text{ matrices }P}} \mathcal{S}_{P, D}\ar[hook]{r} &   \displaystyle{\coprod_{n\times m \text{ matrices }P}} (\mathcal{S}_P, d_\infty)\ar{rr}{\text{id}}\ar{d}[swap]{\Psi} & &
         \displaystyle{\coprod_{n\times m \text{ matrices }P}} (\mathcal{S}_P, d_N)\ar{d}{ 
         \coprod\limits_{n\times m \text{ matrices }P} \text{res}_P}\\
         &
         \left(\text{\rm tame}([0,\infty)^d, \text{\rm vec}_K), d_N\right)
         && \left(\text{\rm Fun}([1], \text{\rm tame}([0,\infty)^d, \text{\rm vec}_K)), d_N\right).\ar{ll}[swap]{\text{\rm coker}}
    \end{tikzcd}
    \]
    Note that 
     \[\mathcal{T}(D,n,m)=\Psi\left(  \displaystyle{\coprod_{n\times m \text{ matrices }P}} \mathcal{S}_{P, D}\right). \]

According to Corollary~\ref{adgjkk} the composition of the right vertical map with the map $\text{coker}$ is $1$-Lipschitz. 
In the case the contour $\ast$ is continuous, the top horizontal right arrow  is continuous. These two facts imply that $\Psi$ is also continuous. Statement (1) of the proposition follows since the image of a compact set via a continuous map is also compact, and  statement (2) 
follows  since $d_N$ is a metric. 
%
\end{proof}

\begin{point}
    If $d_N$ is a metric, compact subsets are closed, so in particular $\mathcal{T}(D,n-1,m)$ and $\mathcal{T}(D,m,n-1)$ are closed subsets of $\mathcal{T}(D,n,m)$. Complements such as $\mathcal{T}(D,n,m)\setminus\mathcal{T}(D,n-1,m)$ as well as $\mathcal{T}(D,n,m)\setminus(\mathcal{T}(D,n-1,m)\cup\mathcal{T}(D,n,m-1))$ which consists of isomorphic tame functors with exactly $n$ generators and $m$ relations, are open subsets of $\mathcal{T}(D,n,m)$ with the relative topology.
\end{point}

We finish with two examples. 
\begin{point}
    Assume that the contour $\ast$ is in dimension $1$ and is the translation of the standard contour by $x=1$ (see~\ref{dsdgdfh}).
    Explicitly, $\ast$ is given by the formula:
     \[v\ast t=\begin{cases}
        \max(v,1)+t, &\text{ if } v<\infty\\
        \infty, & \text{ if } v=\infty.
    \end{cases}\]
    Note that for $v\ast 0=1$  for $v<1$.
    Thus the distance $d_N$ is not  a metric on the set of isomorphism classes of tame functors $\text{tame}([0,\infty),\text{vec}_K)$ (see~\ref{asfhgfgdjgh} and~\ref{asdgadsfth}). In particular 
    $d_N((K(0.5,-), K(0.75,-))=0$.
    
    According to~\ref{asfhgfgjghjkg}.(3), the contour $\ast$ is continuous.
    Thus by~\ref{sadfhfggjh}.(1), the set $\mathcal{T}([0,0.5],1,0)$, which coincides with  $\{K(v,-)\mid v\in D\}$, is compact. However it can not be closed as it does not contain $K(0.75,-)$.
\end{point}
\begin{point}
    Assume $\ast$ is the standard contour in dimension $1$ (explicitly $v\ast t=v+t$).
    Consider the map $[0,1]\to \text{tame}([0,\infty),\text{vec}_K)$,
    $v\mapsto K(v,-)\oplus K(v,-)$, which according to~\ref{asfhgfgdjgh} is an isometry. Its image $\{K(v,-)\oplus K(v,-) \mid v\in [0,1]\}$, with the topology induced by $d_N$, is therefore compact. 
    This image however is not a set of the form $\mathcal{T}(D,n,m)$. 
\end{point}

\newpage
\printbibliography
\end{document}